\documentclass{amsart}
\usepackage{amsmath}
\usepackage{amscd}
\usepackage{amsfonts,latexsym,amssymb}
\usepackage{graphicx}
\usepackage{xcolor}
\usepackage[colorlinks=true]{hyperref}
\usepackage{cleveref}
\usepackage[normalem]{ulem}
\usepackage{bm}
\usepackage{tensor}
\usepackage[all]{xy}

\numberwithin{equation}{section}

\newcommand{\s}{\sigma}

\renewcommand{\o}{\omega}

\renewcommand{\O}{\Omega}

\newcommand{\SA}{{\mathcal{A}}}

\newcommand{\SC}{{\mathcal{C}}}

\newcommand{\SM}{{\mathcal{M}}}

\newcommand{\CrC}{{\mathcal C}}
\newcommand{\CrD}{{\mathcal D}}
\newcommand{\CrE}{{\mathcal E}}

\newcommand{\CrS}{{\mathcal S}}

\newcommand{\Z}{\mathbb{Z}}
\newcommand{\C}{\mathbb{C}}

\newcommand{\N}{\mathbb{N}}
\newcommand{\RR}{\mathbb{R}}

\newcommand{\rd}{\mathrm{d}}

\newcommand{\congm}{\cong}

\newcommand{\Sym}{\operatorname{Sym}}

\newcommand{\im}{\operatorname{im}}

\newcommand{\dimR}{\text{dim}_{\R}\,}

\newcommand{\surj}{\twoheadrightarrow}

\newcommand{\too}{\longrightarrow}

\newcommand{\PD}{\operatorname{PD}}

\newcommand{\SOO}{\mathrm{SO}}
\newcommand{\OO}{\mathrm{O}}
\newcommand{\DD}{\mathbb{D}}
\newcommand{\UU}{\mathrm{U}}
\newcommand{\CC}{{\mathbb C}}
\newcommand{\ZZ}{{\mathbb Z}}
\newcommand{\NN}{{\mathbb N}}
\newcommand{\QQ}{{\mathbb Q}}
\newcommand{\R}{{\mathbb R}}
\newcommand{\orb}{\mathrm{orb}}

\newcommand{\la}{\langle}
\newcommand{\ra}{\rangle}
\newcommand{\sing}{\mathrm{Sing}}
\newcommand{\Obj}{\mathrm{Obj}}
\newcommand{\Mor}{\mathrm{Mor}}

\newcommand{\bd}{\partial}
\newcommand{\bbd}{\bar{\partial}}
\newcommand{\x}{\times}
\newcommand{\ox}{\otimes}

\newtheorem{proposition}{Proposition}[section]
\newtheorem{theorem}[proposition]{Theorem}
\newtheorem{definition}[proposition]{Definition}
\newtheorem{lemma}[proposition]{Lemma}

\newtheorem{corollary}[proposition]{Corollary}
\newtheorem{remark}[proposition]{Remark}
\newtheorem{ex}[proposition]{Example}

\title[Asymptotically holomorphic theory for symplectic orbifolds]{Asymptotically holomorphic theory \\ 
for symplectic orbifolds}

\author[F. Gironella]{Fabio Gironella}
\address{Department of Mathematics, Humboldt University, Berlin, Germany}
\email{fabio.gironella@math.hu-berlin.de, fabio.gironella.math@gmail.com}

\author[V. Mu\~{n}oz]{Vicente Mu\~{n}oz}
\address{Departamento de Algebra, Geometr\'{\i}a y Topolog\'{\i}a, Universidad de M\'alaga,
Campus de Teatinos, s/n, 29071 Malaga, Spain}
\email{vicente.munoz@uma.es}

\author[Z. Zhou]{Zhengyi Zhou}
\address{Morningside Center of Mathematics \& Institute of Mathematics, AMSS, CAS, China}
\email{zhyzhou@amss.ac.cn}

\begin{document}

\maketitle

\begin{abstract}
We extend Donaldson's asymptotically holomorphic techniques to symplectic orbifolds.
More precisely, given a symplectic orbifold such that the symplectic form defines an integer cohomology class,
we prove that there exist sections of large tensor powers of 
the prequantizable line bundle such that their zero sets are 
symplectic suborbifolds. We then derive a Lefschetz hyperplane theorem for
these suborbifolds, that computes their real cohomology up to middle dimension. We also get 
the hard Lefschetz and formality properties for them, when the ambient manifold satisfies those
properties.
\end{abstract}

\section{Introduction} \label{sec:introduction}

In his pioneering work \cite{Do96}, Donaldson introduced the notion of
\emph{asymptotically holomorphic} sections, satisfying a certain \emph{quantitative
transversality} condition with respect to the zero section, on certain complex
line bundles over a given symplectic manifold. His main motivation was the
construction of codimension $2$ symplectic submanifolds belonging to a
cohomology class canonically associated to (a large integer multiple of) the
symplectic form.

Since this foundational work, many authors have further developed the theory
and explored its consequences. First, \cite{Au97} improved the techniques of
Donaldson and worked in the setting of complex vector bundles of higher rank
and with parametric families of almost complex structures. The work \cite{Aur02} also
gave a simplification of a key technical ingredient in \cite{Do96} using
results on complexity of real algebraic sets from \cite{Yom83}. As far as
improvements concerning the properties of the resulting symplectic submanifolds
are concerned, it has been proven in \cite{AGM01} that this can be chosen to avoid an \emph{a priori}
given isotropic submanifold, and \cite{Moh19} proved that additional
transversality conditions along an a priori given submanifold can be
guaranteed. It has been furthermore shown in \cite{Gir17} that the complement of
Donaldson divisors have a Weinstein structure of finite type. In the contact
setting, \cite{IMP99} developed Donaldson's techniques in order to provide
contact codimension $2$ submanifolds with similar properties to Donaldson
divisors (see also the generalization in \cite{IboMarTor04}).

The techniques from \cite{Do96} also paved the way for very important
topological decompositions in symplectic and contact topology. First,
\cite{Do99} proved existence of Lefschetz pencils, whose properties have further
been studied in \cite{AMP04}. Another striking consequence is given in
\cite{ADK05}, where it is shown that, on any oriented smooth closed
$4$-manifold whose intersection form is not negative definite, there is a
``singular'' Lefschetz pencil. Analogous constructions in the case of ``linear
systems'', namely the existence of asymptotically holomorphic sections valued in $\CC P^3$
or in $\CC P^N$ with $N\gg 0$ have been studied respectively in
\cite{Au99,Aur01} and \cite{MPS02}. What's more, \cite{AMP05} also proves that
the pencil can be arranged to induce a Morse function  on a given Lagrangian
submanifold. In the open symplectic manifold setting, \cite{GirPar17} used
these techniques to obtain the existence of Lefschetz fibrations on Stein manifolds
(also on Weinstein manifolds via \cite{CieEliBook}). In the contact setting,
\cite{Gir02} built on the results in \cite{IMP99} in order to show the
existence of supporting open book decompositions for high-dimensional contact
manifolds (see also \cite{PreNotes} for an extension to almost contact manifolds). In a similar direction,
\cite{Pre02} found another kind of topological decomposition for contact
manifolds, namely contact Lefschetz pencils.

Other interesting applications of this powerful theory have also been given in
different directions. For instance, \cite{FM} proved formality of
Donaldson submanifolds. Moreover, in the foliated setting, \cite{OMP05} used
these techniques to study (complex) codimension $1$ (singular) symplectic
foliations, whereas \cite{MdPP18} used them to understand (real) codimension
$1$ strong symplectic foliations.

In this paper, we explore the use of asymptotically holomorphic techniques in
the symplectic orbifold setting.
The motivation to do so comes from the recent traction which the study of symplectic orbifolds has been gaining. 
Indeed, naturally appearing in relation to the Mirror Conjecture (see e.g.\ \cite{MR1115626}), orbifold symplectic geometry has recently been used in order to get very interesting results in the smooth symplectic setting. 
For instance, Lagrangian Floer theory developed in \cite{orbiLagrangian} has recently been utilized in \cite{mak2021non} to obtain new families of non-displaceable Lagrangian links in symplectic four-manifolds. 
This idea was further explored in \cite{polterovich2021lagrangian,cristofaro2021quantitative} which brought breakthroughs on dynamics on surfaces and $C^0$ symplectic geometry. 
Symplectic orbifolds were also used in \cite{gironella2021exact} to study the symplectic cobordism category of contact {manifolds}. From yet another point of view, they also appear in constructions of symplectic manifolds via the desingularisation process. This was introduced in \cite{FM-Annals} to construct new non-formal symplectic manifolds, and later developed in \cite{CFM, NiPa, MuRo, Chen, MMR}.

It is hence a natural direction of research that of extending well known and powerful techniques in smooth symplectic geometry to the orbifold setting.
The first set of tools are pseudo-holomorphic and Floer theories, which have been the object of study for instance in \cite{CR01,orbiLagrangian,gironella2021exact}. 
This paper is devoted to extend the techniques in \cite{Do96} to the orbifold setting.
Applications in terms of geometric decompositions of symplectic/contact orbifolds, in analogy with the smooth case recalled above, will be the object of future work.
\\

In order to state the main result of this work, let us
introduce some notations first.
Let $(X,\omega)$ be a symplectic orbifold with 
$[\omega/2\pi]\in \bar H^2(X,\ZZ)$,
where $\bar H^2(X,\ZZ)$ denotes the image of the 
singular cohomology of the topological space $X$ in the de Rham cohomology.
We fix a compatible almost complex structure $J$, and let $g$ be the associated
Riemannian metric $g=\omega(\cdot,J\cdot)$. Consider now a Hermitian complex
line bundle $L\to X$ with $c_1(L)=[\omega/2\pi]$, and a connection $\nabla$ on it,
with curvature $F_{\nabla}= -i\omega$. 
We also consider the tensor power $L^{\otimes k}$, for $k\geq 1$.
This has an induced connection, again denoted by $\nabla$ with a little abuse
of notation, whose curvature is $F_{\nabla} =-i \omega_k=-ik\,\omega$, where
$\omega_k=k\,\omega$ is the rescaled symplectic form.

\begin{theorem}
\label{thm:existence_intro}
For $k\gg 0$, there exists an asymptotically holomorphic sequence of sections $s_k$ of $L^{\otimes k}$ that is $\eta$-transverse to $0$, for some $\eta>0$ independent of $k$. 
In particular, $s_k^{-1}(0)$ is a symplectic suborbifold.
\end{theorem}

The notions of ``asymptotically holomorphic'' and ``$\eta$-tranverse'', already appearing in \cite{Do96}, will be made precise in \Cref{sec:donaldson}.
For the time being, an asymptotically holomorphic sequence of sections $s_k$ can be understood as a sequence of sections $s_k$ of $L^{\otimes k}$ with $|\overline{\partial}_J s_k|$ arbitrarily small, provided that $k$ is big enough.
Moreover, $\eta$-transversality can be thought of as a quantitative refinement of usual transversality.
The second part of the statement follows from the first part by an easy argument using the fact that the almost complex structure $J$ is adapted to $\omega$ (see \Cref{prop:sympl_suborb_from_sequence}).

The proof of the first part is on the contrary more involved, and follows the same line of argument as in \cite{Do96}.
However, there is an additional difficulty, coming from the orbifold setup and consisting in the following:
for those points which are very near the isotropy locus, one  needs to achieve transversality via equivariant sections, defined in a local orbifold chart centered at a nearby isotropy point, and the perturbation lemma as in \cite{Do96} is no longer sufficient (for points which are not on the isotropy locus itself).
Here, we propose the following solution to this issue. 
First, one proves, via an inverse inductive argument on the ``height'' of the orbifold stratum (i.e.\ starting from the biggest strata and going to the smaller ones), that there is a ``convenient" lattice of points in each stratum, away from a neighborhood (of a certain specific size) of the strata strictly contained in it. 
Then, one proves that transversality can be achieved everywhere, again by inverse induction on the height, by perturbing around points of these lattices in each stratum (using the perturbation lemma in \cite{Do96}). 
The subtle point is of course showing that this inductive procedure indeed results in achieving transversality everywhere. 
This follows from the special properties of the previously found lattice (which we naively summarized in the previous sentence with the adjective ``convenient''), namely from the following property: local perturbations at each of the points of the lattice in a given stratum allows to achieve transversality at least up to the region which is missed by the lattices in the bigger strata. 

\begin{remark}
\label{rmk:almost_hol}
As it has been pointed out to us by Steven Zelditch, it is worth pointing out that, in the setting of a smooth manifold $X$, there is an alternative approach to constructing sections of $L^{\otimes k}\to X$ which are, for $k$ increasing, ``more and more holomorphic''.
Following ideas in \cite{Bou74,BouGui81}, this approach has been carried out in detail in \cite{ShiZel02}, where \emph{almost holomorphic sections} are studied, which are essentially the kernel of a zero-th order perturbation of the operator $\overline{\partial}_L$. 
The authors then prove how to achieve quantitative transversality to the zero section for almost holomorphic sections adapting the arguments in \cite{Do96,Au97}, thus also obtaining the existence of symplectic divisors via this alternative point of view.
One could carry out the theory of almost holomorphic sections in the orbifold setting as well with minimal modifications; the difficulty in the orbifold case still resides however in achieving transversality.

A further point worth stressing out, already in the smooth case, is the fact that, despite transversality is a generic condition in the smooth world, asymptotically holomorphic (and also almost holomorphic) sections which are transverse to the zero section seem to be a ``rare'' phenomenon, with their existence being a very deep result and requiring extremely non-trivial arguments.
\end{remark}

After \Cref{thm:existence_intro} is proved, we explore its consequences.
First, we point out an obvious corollary consisting in the existence of invariant Donaldson submanifolds in the setting of finite group actions by symplectomorphisms on symplectic manifolds.

We then prove the following homological version of the Lefschetz hyperplane theorem for symplectic orbifolds:

\begin{theorem}
\label{thm:homology_lefschetz_hyperplane_intro}
Let $(X,\omega)$ be a symplectic orbifold of dimension $2n$, and $J$ a compatible almost complex structure on it. 
Consider then a sequence of asymptotically holomorphic
 sections $s_k$ of $L^{\ox k}$ which is $\eta$-transverse to zero, and let $Z_k=Z(s_k)$ be their zero sets. 
 Then for $k\gg 0$
 $$
   \begin{aligned}
   H_i(Z_k,\RR) \cong H_i(X,\RR), \text{ for }i\leq n-2, \\
  H_i(Z_k,\RR) \twoheadrightarrow H_i(X,\RR), \text{ for }i= n-1.
  \end{aligned}
$$
\end{theorem}

In analogy with the smooth statement from \cite{Do96}, one could wonder what happens for the orbifold homotopy groups.
In this paper, we prove surjectivity of the map induced at the level of the orbifold fundamental group (see 
\Cref{thm:lefschetz_hyperplane_fund_grp}). We also extend this result to the case of the homotopy groups ignoring torsion (see \Cref{thm:homotopy_lefschetz_hyperplane}).

A second consequence we explore is related to the \emph{hard Lefschetz property}. This was proved for symplectic manifolds in \cite{FM}.
A symplectic orbifold $(X,\omega)$ of dimension $2n$ is said to satisfy this property if  
  $$
   [\omega]^{n-i}: H^i(X,\RR) \too H^{2n-i}(X,\RR)
  $$
is an isomorphism for all $i\leq n-1$. 

\begin{theorem} \label{hardLefschetzDonaldson_intro}
Let $X$ be a compact symplectic orbifold of dimension $2n$, and
let $Z\subset X$ be a Donaldson suborbifold (i.e.\ the zero set of a section as in \Cref{thm:existence_intro}, for $k$ big enough). 
If $X$ satisfies the
hard Lefschetz property, then $Z$ does too.
\end{theorem}

We also draw some conclusions regarding \emph{formality} of the asymptotically holomorphic suborbifold in terms of the ambient symplectic orbifold. 
Formality is a property of the real homotopy type of a space that says that the real homotopy groups are completely determined by its real cohomology rings (in principle, for simply connected or nilpotent spaces, see \Cref{sec:formality} for precise definitions).
K\"ahler manifolds and K\"ahler orbifolds are formal according to, respectively, \cite{DGMS} and \cite{BBFMT}, whereas symplectic manifolds are not formal in general \cite{TO}. 
One striking result was the construction of a simply connected symplectic $8$-manifold which is not formal \cite{FM-Annals} using symplectic resolution of singularities of symplectic orbifolds. 
The formality property of asymptotically holomorphic divisors in symplectic manifolds was studied in \cite{FM}. 
This was done by a generalization of the notion of formality, namely \emph{$s$-formality}. 
Here we extend the result to the orbifold case.

\begin{theorem}
\label{thm:formality}
Let $X$ be a compact symplectic
orbifold of dimension $2n$ and let $Z\subset X$ be a Donaldson suborbifold. 
For each $s\leq n-2$, if $X$ is $s$-formal then $Z$ is $s$-formal. In particular,
$Z$ is formal if $X$ is $(n-2)$-formal.
\end{theorem}

As mentioned to us by Steven Zelditch, a last property of the sequence of sections $s_k$ constructed in \cite{Do96}, and that also holds in the orbifold setting of \Cref{thm:existence_intro} with analogous proof (here omitted), is the fact that their zero sets get more and more ``equidistributed'' in the ambient manifold as $k$ increases.
This property can be precisely formulated in terms of currents; see \cite[Proposition 40]{Do96} and \cite[Theorem 7]{Sen06}.

\subsection*{Outline}
\Cref{sec:orbifolds} recalls all the definitions and properties of orbifolds which we will need in the sections afterwards.
More precisely, both the classical and groupoidal approaches are utilized, the first being more geometrical and intuitive in nature, and the latter being better suited to the description of orbifold stratification that is needed for the induction argument in the proof of the main result.

In \Cref{s:geo} we recall the needed topological/geometric notions on orbifolds, such as fundamental group, differential forms, almost complex structures, and (orbifold) complex line bundles.

\Cref{s:lattice,sec:donaldson} together contain the proof of \Cref{thm:existence_intro} on the existence of asymptotically holomorphic sequence of sections which are quantitatively transverse to the zero section.
More precisely, \Cref{s:lattice} explains how to find a lattice of points in the ambient orbifolds. 
\Cref{sec:donaldson} then shows how the asymptotically holomorphic local peak sections around each of these points can be perturbed to globally achieve quantitative transversality.
At the end of the section, we also point out the obvious application of \Cref{thm:existence_intro} to the case of finite actions by symplectomorphisms on symplectic manifolds.

Lastly, \Cref{sec:lefschetz,sec:hard_lefschetz_property,sec:formality} contain the applications of \Cref{thm:existence_intro}, namely the proofs
of \Cref{thm:homology_lefschetz_hyperplane_intro} on the homology version of the Lefschetz hyperplane theorem for Donaldson submanifolds, 
(a more detailed version of) \Cref{hardLefschetzDonaldson_intro} concerning their hard Lefschetz property, as well as \Cref{thm:formality} on their formality properties.

\subsection*{Acknowledgements}
The authors are very grateful to Fran Presas for useful discussions concerning the subtleties of the perturbation near the isotropy locus.
We also thank warmly to Nieves Alamo for her contribution in the initial steps of this paper.
Lastly, we would also like to thank Steven  Zelditch for useful feedback on the first version of the preprint, concerning \Cref{rmk:almost_hol} and the equidistribution property for the zero sets of the sections.
The first author is supported by the European Research Council (ERC) under the European Union’s Horizon 2020 research and innovation programme (grant agreement No. 772479).
The second author is partially supported by Project MINECO (Spain) PID2020-118452GB-I00.

\section{Orbifolds}\label{sec:orbifolds}

\subsection{Geometric orbifolds}
Let $X$ be a topological space and $n\,>\,0$.
An {\em orbifold chart} $(U, {\widetilde U}, H, \varphi)$ on $X$
consists of an open set $U \,\subset\, X$, a connected and open set
${\widetilde U} \,\subset\, \RR^n$, a finite group $H< \OO(n)$ acting linearly
on $\widetilde{U}$, and a continuous map
$$
\varphi\,\colon\,\widetilde{U}\,\longrightarrow\,U,
$$
which is  $H$-invariant
(that is $\varphi\,=\,\varphi\circ h$, for all $h \in H$) and such that it induces a homeomorphism
$$
{\widetilde U}/H \stackrel{\cong}{\longrightarrow} U.
$$

\begin{definition}\label{def:geo_orbifold}
An \emph{orbifold} $X$ of dimension $n$, is a Hausdorff, paracompact topological
 space endowed with an equivalence class of orbifold atlases. Here, an {\em orbifold atlas} $\SA=\{(U_i, {\widetilde U}_i, H_i, \varphi_i)\}$ is a family of orbifold charts with $X=\bigcup U_i$, and such that
 \begin{enumerate}
 \item[i)] if $(U_i, {\widetilde U}_i, H_i, \varphi_i)$ and $(U_j, {\widetilde U}_j, H_j, \varphi_j)$
 are two orbifold charts, with $U_i\,\cap\,U_j\,\not=\emptyset$, then for each point $p\,\in\,U_i\,\cap\,U_j$ there exists
an orbifold chart $(U_k, {\widetilde U}_k, H_k, \varphi_k)$ such that $p\,\in\,U_k\,\subset\,U_i\,\cap\,U_j$;
 \item[ii)] if $(U_i, {\widetilde U}_i, H_i, \varphi_i)$ and $(U_j, {\widetilde U}_j, H_j, \varphi_j)$
 are two orbifold charts, with $U_i\,\subset\,U_j$, then there exists a smooth open embedding, called {\em change of charts}
 $\rho_{ij}\colon{\widetilde U}_i\rightarrow{\widetilde U}_j$ such that $ \varphi_i\,=\,\varphi_j\,\circ\,\rho_{ij}$.
 \end{enumerate}
 Two orbifold atlases are moreover \emph{equivalent} if their union is also an orbifold atlas.
\end{definition}

Let $X$ be an orbifold, and $p\in X$. Consider $(U_i, {\widetilde U}_i, H_i, \varphi_i)$ an orbifold
chart around $p$, that is $p= \varphi_{i}(x)\in U_i$ with $x\,\in\,{\widetilde U}_i$,
and denote by ${H_i}(x)\,\subset H_i$ the isotropy
subgroup for the point $x$. ${H_i}(x)$ does not depend
on the choice of the orbifold chart around $p$. The group ${H_i}(x)$ is called  the
{\em isotropy group} of $p$, and it is denoted by $H_p$. When $H_p$ is not trivial, the point $p$ is said to
be a {\em singular} point of the orbifold $X$.
The points $p$ with $H_p$ trivial are called {\em regular} points.
The set of singular points
$$\sing(X):=\{p \in X\,\vert\, H_p \text{ is not trivial}\}$$
is called the {\em singular set} of the orbifold $X$.
Then $X\setminus \sing(X)$ is a smooth $n$-dimensional manifold.

\begin{remark}\label{rem:2.2}
    By definition, the local action in the orbifold charts is effective, as the local group is a subgroup of $\OO(n)$.
    For this reason, these are also called \emph{effective} or \emph{reduced}, orbifolds. 
    The groupoidal approach to orbifolds as in \cite{Adem} 
    (see Section \ref{ss:groupoid} below) allows also to study unreduced orbifolds, where $\sing(X)$ can be the whole of $X$.
    There is a standard procedure to obtain a canonical effective orbifold from a general (possibly unreduced) orbifold \cite[Section 7.2]{polyfold}. 
\end{remark}

In view of \Cref{rem:2.2}, when the distinction needs to be made, we 
call \emph{geometric orbifolds} the ones in Definition \ref{def:geo_orbifold}, and \emph{groupoidal orbifolds} the ones in 
Section \ref{ss:groupoid}, although they are equivalent in the effective case (see \cite[Proposition 1.44]{Adem}). We will formulate various geometric structures on orbifolds in
Section \ref{s:geo}, using Definition \ref{def:geo_orbifold}, but use the groupoidal language to discuss the stratification in 
Section \ref{ss:stratification}.

\begin{definition}[\cite{BG}] \label{def:orbimap}
Let $X,Y$ be two orbifolds 
¡and let $\{(U_i, {\widetilde U}_i, H_i, \varphi_i)\}$ and $\{(V_j, {\widetilde V}_j, K_j, \psi_j)\}$ be
atlases for $X$ and $Y$, respectively.
A map $f\colon X \rightarrow Y$ is said to be an {\em orbifold map} 
if $f$ is a continuous map between the underlying topological spaces, and
for every point $p\in X$ there are orbifold charts $(U_i, {\widetilde U}_i, H_i, \varphi_i)$
and $(V_i, {\widetilde V}_i, K_i, \psi_i)$ around $p$ and $f(p)$ respectively, with $f(U_i)\subset V_i$, a
differentiable map ${\widetilde f}_i\colon {\widetilde U}_i\rightarrow {\widetilde V}_i$, and a
homomorphism $\varpi_i:H_i\to K_i$ such that ${\widetilde f}_i \circ h=\varpi_i(h)\circ
{\widetilde f}_i$ for all $h\in H_i$, and
$$
f_{|\,U_{i}}\circ\varphi_i=\psi_i\circ{\widetilde f}_i.
$$
Moreover, we suppose that 
every map ${\widetilde f}_i$ is
{\em compatible with the changes of charts}:
\begin{enumerate}
\item[i)] if $\rho_{ij}\colon{\widetilde U}_i\,\rightarrow\,{\widetilde U}_j$ is a change of charts
around $p$, then there is a change of charts $\mu(\rho_{ij})\colon{\widetilde V}_i\,\rightarrow\,{\widetilde V}_j$ around $f(p)$ such that
${\widetilde f}_j\circ\rho_{ij}\,=\,\mu(\rho_{ij})\circ{\widetilde f}_i$, and
\item[ii)]  if $\rho_{ki}\colon{\widetilde U}_k\,\rightarrow\,{\widetilde U}_i$ is a change of charts
around $p$, then $\mu(\rho_{ij}\circ \rho_{ki})\,=\, \mu(\rho_{ij})\circ \mu(\rho_{ki})$.
\end{enumerate}
Therefore, an orbifold map  $f\colon X\rightarrow Y$ is determined by a smooth map
${\widetilde f}_i\colon{\widetilde U}_i\rightarrow{\widetilde V}_i$, for every orbifold chart   $(U_i, {\widetilde U}_i, H_i, \varphi_i)$
on $X$, such that every
${\widetilde f}_i$ is $H_i$-equivariant and compatible with the change of orbifold charts.
\end{definition}

Observe that the composition of orbifold maps is an orbifold map. Moreover,  an orbifold map $f\colon X\rightarrow Y$  induces a homomorphism
from $H_{p}$ to $K_{f(p)}$.

\begin{remark}
Notice that the notion of orbifold map defined in \Cref{def:orbimap} corresponds to the notion of \emph{good maps} in \cite{CR01}, which is equivalent (at least for effective orbifolds) to generalized maps (generalizing functors) defined using groupoids in \cite{MP,M02} by \cite{lupercio2004gerbes}. This notion of maps is sufficient to talk about many geometric constructions, e.g.\ the pullback of orbifold vector bundles \cite[Section 5.1]{M02}. Moreover, the set of orbifold maps of certain regularity (e.g.\ $C^k$, $W^{k,p}$) can be endowed with a topology, such that it becomes a Banach orbifold, see \cite{Chen06,gironella2021exact}.
\end{remark}

\subsection{Groupoidal orbifolds}\label{ss:groupoid}

We now describe a natural stratification of an orbifold coming from the different (isomorphism types of) isotropy groups of each point, which 
 will be needed in the proof of the main result of this paper. The closure of each stratum is not typically a suborbifold or an embedded orbifold, but rather the image of immersed orbifolds.
We describe this using the groupoid point of view for orbifolds. To this end, we start by recalling the basic notions needed for our purposes, and we refer the interested reader to the more comprehensive references \cite{Adem,M02} for details and additional explanations.

\begin{definition}
\label{def:groupoid}
	A \emph{proper \'etale  Lie groupoid} $\CrC$ is a groupoid (i.e.\ a small category where every arrow is an isomorphism) with $\Obj(\CrC)=C_0$ and $\Mor(\CrC)=C_1$, such that:
	\begin{enumerate}
		\item $C_0,C_1$ are both Hausdorff spaces locally modeled on $\RR^n$ with smooth transition maps.
		\item (\'Etale) The source and target maps $s,t:C_1\to C_0$ are local diffeomorphisms.
		\item The inverse map $i:C_1\to C_1$, unit map $u:C_0\to C_1$ and multiplication $m:C_1\tensor[_s]{\times}{_t} C_1 \to C_1$ are smooth. 
		Recall that $C_1\tensor[_s]{\times}{_t} C_1 := \{(\phi,\psi)\in C_1\times C_1\vert s(\phi)=t(\psi)\}$.
		\item (Proper) $(s,t):C_1\to C_0 \times C_0$ is a proper map. 
	\end{enumerate}
\end{definition}

Given a proper \'etale Lie groupoid $\CrC$, the \emph{orbit set} $|\CrC|=C_0/C_1$, i.e.\ the set of equivalence classes with the equivalence relation $x\sim y$ if
$\phi(x)=y$ for a $\phi\in C_1$, is equipped with the quotient topology and is a Hausdorff space.
Moreover, for a point $x\in C_0$, the \emph{isotropy group} of $x$ is given by $C_x=\{g\in C_1\vert s(g)=t(g)=x\}$. Moreover, $C_x$ is isomorphic to $C_y$ if $x\sim y$, i.e.\ the isotropy group $C_p$ for $p\in \vert \CrC \vert$ is well-defined.

A functor between \'etale proper Lie groupoids is called \emph{smooth} if it is smooth both on the object and morphism levels. 
An \emph{equivalence} from $\CrC$ to another proper \'etale Lie groupoid $\CrD$ is a fully faithful functor $\phi$ that is a local diffeomorphism on the object level and such that $|\phi|:|\CrC|\to \vert\CrD\vert$ is a homeomorphism. 
In what follows, we will also denote by $\phi_0$ and $\phi_1$ the functor $\phi$ at the level of objects and morphisms respectively.
Lastly, if there is a diagram of smooth equivalences
 \[ 
\CrC \stackrel{\sim}{\leftarrow}\CrE \stackrel{\sim}{\to }\CrD,
 \]
then $\CrC$ and $\CrD$ are called \emph{Morita equivalent}. 

\begin{definition}
\label{defn:orbifold}
	An \emph{orbifold structure} $(\CrC,\alpha)$ on a paracompact Hausdorff space $X$ is a proper \'etale Lie groupoid $\CrC$ with a homeomorphism $\alpha\colon|\CrC|\to X$. 
	Two orbifold structures $(\CrC,\alpha),(\CrD,\beta)$ are said to be \emph{equivalent} if there is a Morita equivalence $\CrC\stackrel{\mathfrak{f}}{\to} \CrD$ such that $\alpha=\beta\circ |\mathfrak{f}|$.
	An \emph{orbifold} $X$ is a paracompact Hausdorff space $X$ equipped with an equivalence class of orbifold structures $(\CrC,\alpha)$.
\end{definition}

\begin{ex}
Given an action of a finite group $G$ on $M$, the \emph{translation groupoid} $G\ltimes M$ is defined as follows.
The set of objects is $M$, and the set of morphisms is $G\times M$.
The source and target maps are $s,t: G\times M \to M$, $s(g,m)=m$ and $t(g,m)=g\cdot m$, where $g\cdot m$ denotes the action of $g\in G$ on $m\in M$. If $G$ acts on $M$ effectively, then we can view the quotient $M/G$ as a geometric orbifold, or equivalently, we can view the translation groupoid $G\ltimes M$ as the orbifold structure equipped on the quotient space $M/G$. Since a geometric orbifold is locally modeled on such quotient spaces, i.e.\ the charts, we have the analogous structure for the groupoidal description of orbifolds in the following definition.
\end{ex}

\begin{definition} \label{defn:loc_unif}
	Let $\CrC$ be a proper \'etale Lie groupoid and $x\in C_0$. 
	A \emph{local chart/uniformizer} around $x$ is a smooth and fully faithful functor 
	\[\Psi_x: C_x\ltimes U_x\to \CrC,\]
	with $U_x\subset C_0$ a neighborhood of $x$ and $C_x\subset C_1$ the isotropy group of $x$, such that the following holds:
	\begin{enumerate}
		\item\label{inclusion} On the objects level, $\Psi_x$ is the inclusion $U_x\to C_0$,
		\item\label{orbit_homeo} $|\Psi_x|:U_x/C_x\to |\CrC|$ is a homeomorphism onto an open subset of $|\CrC|$.
	\end{enumerate}
\end{definition}

For any groupoidal orbifold and any point, we can always find a local uniformizer around that point \cite[Proposition 7.1.19]{polyfold}. However, the $C_x$ action on $U_x$ does not have to be effective. We will say that an orbifold $X$ is \emph{locally modeled} on $U_x/C_x$ if $X$ is represented by a proper \'etale Lie groupoid $\CrC$ and $C_x\ltimes U_x$ is a local chart around $x\in C_0$ for $\CrC$.

Due to the additional data of isotropy groups, there are different notions of suborbifolds with different requirements on the compatibility of isotropy groups, c.f.\ \cite[Definition 7.1.21]{polyfold} and \cite[Definition 2.3]{Adem}, see also \cite{BorBru15} for the comparison. As it turns out, our stratification of orbifolds is induced from orbifold embeddings/immersions defined as follows.

\begin{definition}[{\cite[Definition 2.3]{Adem}}]\label{def:orbifold_embedding}
A smooth functor $\phi:\CrC\to \CrD$ between étale Lie groupoids is said to induce a \emph{(proper) embedding} if the following holds:
\begin{enumerate}
    \item\label{item1:embedding} $\phi_0:C_0\to D_0$ is an immersion.
    \item\label{orbifold_embedding} Let $y\in \im (\phi_0)$ and $D_y\ltimes V_y$ a local uniformizer around $y$. 
    Then the $\CrC$-action on $\phi_0^{-1}(y)$ is transitive, and there exists an open neighborhood $U_x\subset C_0$ of every $x\in \phi^{-1}_0(y)$ such that $\CrC|_{U_x}=C_x\ltimes U_x$, 
    $C_x$ is mapped injectively into $D_y$ by $\phi_1$, and
    $$\CrC|_{\phi^{-1}_0(V_y)}\simeq D_y\ltimes (D_y\times U_x)/C_x,$$
    where $\simeq$ stands for Morita equivalence.
    \item\label{item3:embedding} $|\phi|:|\CrC|\to |\CrD|$ is proper. 
\end{enumerate}
A smooth functor $\phi:\CrC\to \CrD$ is called \emph{a (proper) immersion} if \eqref{item1:embedding} and \eqref{item3:embedding} hold as above, and \eqref{orbifold_embedding} holds locally on $\CrC$.
\end{definition}

\subsection{A stratification on orbifolds}\label{ss:stratification}

Let now $\CrC$ be a proper \'etale Lie groupoid and $H$ a finite group.
We define the translation groupoid $\CrC_H$ 
as follows. The set of objects $(C_H)_0$ consists of pairs $(x,K)$, where $x \in C_0$ and $K<C_x$ is a subgroup which is isomorphic to $H$. 
A morphism from $(x,K)$ to $(x',K')$ consists of a morphism $g: x \to x'$ in $C_1$ such that $K'=g K g^{-1}$ as a subgroup of the isotropy group at $x'=gx$.

In other words, $\CrC_H$ is a groupoid describing \emph{representable} maps (in the sense of \cite[Definition 2.44]{Adem}, i.e.\ inducing injective maps between stabilizers) from $\bullet/H$ to $\CrC$, modulo the equivalences from the automorphisms (reparameterizations) of $H$, i.e.\ this can be viewed as the \emph{space of $H$-points}. Notice that when $H=\{1\}$, then simply $X_{\{1\}}=X$. The following shows that $\CrC_H$ is a proper \'etale Lie groupoid.

\begin{proposition}\label{prop:C_H_orbifold}
$\CrC_H$ is a proper \'etale Lie groupoid. 
More precisely, if $\CrC$ has local chart on $G\ltimes \RR^n$ at a point $x\in C_0$, for some representation $G\to \OO(n)$, then near a point $(x,K)$, $K\cong H$, in the set of objects, $\CrC_H$ has local chart $N_G(K)\ltimes \RR^n_K$, where $N_G(K)$ is the normalizer of 
$K<G$ and $\R^n_K$ is the subspace 
fixed by $K$. Moreover, the natural map $\CrC_H\to \CrC$ is an immersion. 
\end{proposition}

\begin{proof}
We first describe the manifold structure on $(C_H)_0$ near $(x,K)$, i.e.\ a local chart 
\[
\phi_{(x,K)}\colon \RR^n_K \to (C_H)_0\, ,
\]
whose image contains $(x,K)$.
For $x'$ in the local chart $\RR^n$ of $C_0$ centered at $x$ as in the statement, there is a natural inclusion of the isotropy $C_{x'}$ into  $C_x=G$.
Moreover, if $x'\in \RR^n_K$, then  the inclusion $K\subset C_x$ factors as $K\subset C_{x'}\subset C_x$.
Then, we can define the map $\phi_{(x,K)}\colon \RR^n_K \to (C_H)_0$ as $\phi_{(x,K)} (x') = (x',K)$, where $K$ is seen as subgroup of $C_{x'}$.

The collection of these charts can be easily seen to give to $(C_H)_0$ the structure of a smooth manifold.
As $\CrC_H$ is by definition just the translation groupoid associated to the action of the étale Lie groupoid $\CrC$ on $\CrC_H$, it is then an étale Lie groupoid.
Now, it is clear that the isotropy of a point $(y,K\subset C_y)$ in $\CrC_H$ is just given by those $g\in C_y$ such that $gKg^{-1}=K$.
It follows that the étale proper Lie groupoid $\CrC_H$ is locally modeled near $(x,K)$ on $N_G(K)\ltimes \RR^n_K$ as desired.

Lastly, the natural immersion $\CrC_H\to \CrC$ is given by the union of local embeddings $N_G(K)\ltimes \R^n_K \to G\ltimes \R^n$. 
\end{proof}

Given an orbifold $X$, we denote by $X_H$ the orbifold modeled on $\CrC_H$ for any orbifold structure $\CrC$ of $X$. 
It is straightforward to check this definition is independent of the orbifold structure up to Morita equivalence. 
Notice also that, via the immersion $\CrC_H\to \CrC$, a symplectic form on $X$ pulls-back to one on $X_H$.

We say that a connected component $\tau$ of $\vert \CrC_H\vert$ is \emph{effective} if it contains a point of the form
$\vert(x,G_x)\vert$, i.e.\ a point where the subgroup of the isotropy group is in fact the whole isotropy.

\begin{proposition}\label{prop:opendense}
Let $\tau\subset X_H$ be an effective component.
Then, there is an open dense subset of $\tau$ made of 
points of the form $\vert(x,C_x)\vert$.
\end{proposition}

\begin{proof}
Let $\vert(x,C_x)\vert$ be a point in $\tau$.
Then, by construction of the local chart $\phi_{(x,K)}$ in the proof of \Cref{prop:C_H_orbifold}, in a neighborhood of $(x,C_x)$ inside $\tau$ there are points $(y,C_y=C_x)$.
In particular, the set of points of the type $(z,K=C_z)$ is an open and non-empty subset of $\tau$.

In order to prove density, let
$(x',K')$ be a point of $\tau$  with $K'$ a strict subgroup of $C_{x'}$.
Assume that $X$ and $\tau$ are locally modeled on $\RR^n/G$ and $\RR^n_K/N_G(K)$ respectively, for some $K\subsetneq G$, $K\cong H$. 
The set of points $y$ in $\RR^n_K$ such that 
$K$ is \emph{not} the whole isotropy subgroup of $y$ is just  
 $$
 \left\{\,y\in \RR^n_K \;\left|\; \exists g\in G\setminus K, \, gy=y \,\right.\right\}=\bigcup_{g\in G\setminus K} \RR^n_{\langle g\rangle}\cap \RR^n_K\, ,
 $$
which is either $\RR^n_K$ or a closed subset without interior points. 

Now consider the set of points 
$(x',K')$
which admit an open neighborhood such that every point in it has 
group strictly smaller than the isotropy subgroup.
By its very definition, this set is open.
Moreover, the discussion in the previous paragraph also implies that it is also closed. 
Hence, if it was non-empty, it would be the whole connected component $\tau$. 
In other words, such set needs to be empty, thus proving the desired statement.
\end{proof}

Let $\CrS(X)$ denote the set of all effective connected components $\tau$ of $X_H$ for all $H$. 
We use $H_{\tau}$ to denote the underlying group $H$ for the component $\tau$, and $X_{\tau}$ to denote the image of the natural immersion $\tau \to X$. 
Then we define a relation on $\CrS(X)$ as follows: $\tau \le \xi$ if $X_{\tau} \subset X_{\xi}$.  

\begin{ex}\label{ex:XH_stratification}
Let $X$ be the quotient orbifold $\C^n/G$ for $G < \UU(n)$. 
Given a vector subspace $V\subset\CC^n$, we use $G_{V}$ to denote the maximal subgroup of $G$ that fixes $V$.
Let now $H$ be a subgroup of $G$ such that $G_{\C^n_H}=H$. 
Then $\CrS(X)$ is the set of connected components of the orbifold $X_H$ coming from the ones of the representing \'etale Lie groupoid $\CrC_H$ described in the proof of \Cref{prop:C_H_orbifold}, for all such $H$.
\end{ex}

\begin{proposition}\label{prop:po}
For a compact orbifold $X$, $(\CrS(X),\le )$ is a finite poset. 
\end{proposition}

\begin{proof}
From the local description in \Cref{ex:XH_stratification} and the compactness of $X$, we know that $\CrS(X)$ is a finite set. To verify that $\le $ defines a partial order, it suffices to show that, if $\tau\le \xi$ and $\xi \le \tau$, then $\tau=\xi$. 

The assumption implies $X_{\tau} = X_{\xi}$, and, as both $\xi,\tau$ are effective, one gets $H_{\tau}=H_{\xi}$. 
To verify that they are the same, it suffices to prove that $\tau=\xi$ in a local chart near 
$(x,K =C_x)$.
This is clearly the case from the local description in (the proof of) \Cref{prop:C_H_orbifold}.
\end{proof}

Given any orbifold $X$, there is a procedure for constructing an effective orbifold $X_{\rm{R}}$ with the same underlying quotient \cite[Section 7.2]{polyfold}. On the isotropy group level, the construction takes the quotient of the isotropy group $C_x$ at $x$ by the kernel $H_0$ of the group homomorphism $C_x\to \mathrm{Diff}(\RR^n)$ in the local uniformizer near $x$. 
When $X$ is connected, $H_0$ is the same for any $x$; in this case, we denote it by $H_{0,X}$. 
Obviously, when $X$ is effective, $H_{0,X}$ is trivial. 
We point out that, even when the starting orbifold $X$ is effective, the construction of $X_H$ yields non-effective orbifolds.

\begin{proposition}\label{prop:le_property}
The poset $(\CrS(X),\le )$  has the following properties:
\begin{enumerate}
    \item\label{le_1} If $\tau\le\xi$, then $H_{\xi}$ is a subgroup of $H_{\tau}$.  
    If moreover $H_{\xi}=H_{\tau}$, then $\tau=\xi$.
    \item\label{le_2} If $x\in X_{\tau}\cap X_{\xi}$, there exists $\eta$ such that $x\in \eta$, $\eta \le \tau$ and $ \eta\le\xi$.
    \item\label{le_3} Let $x\in X_{\tau}$ such that the immersion $\tau \to X$ fails to be injective over the point $x$.
    Then, there exists $\xi<\tau$ with $x\in X_{\xi}$.
    \item\label{le_4} Assume $X$ is connected. Then $\CrS(X)$ has a unique maximal element $\tau_{\max}=X$ with $H_{\tau_{\max}}=H_{0,X}$. Moreover, $X\backslash \bigcup\limits_{\tau<\tau_{\max}} X_{\tau}$ is a manifold with a trivial $H_{0,X}$-action.
\end{enumerate}
\end{proposition}

\begin{proof}
\eqref{le_1} Since $\tau$ is effective, there is a point $x\in X_{\tau}$ such that $H_{\tau}=C_x$. 
Now, as $x\in X_{\xi}$ by assumption, $H_{\xi}$ is a subgroup of $C_x=H_{\tau}$. 
If moreover we have $H_{\xi}=H_{\tau}$, we then know that $\tau=\xi$ on a neighborhood $x$ by the local chart picture in \Cref{prop:C_H_orbifold}, hence $\tau$ and $\xi$ are the same component of $X_{H_{\tau}}$.

\eqref{le_2}
By the assumption $x\in X_{\tau} \,\cap\, X_{\xi}$, we have two inclusions $\psi_{\tau}:H_{\tau}\to C_x$ and $\psi_{\xi}:H_{\xi}\to C_x$. 
Let $\eta$ be the connected component of $X_{C_x}$ containing 
$(x,C_x)$.
Using $\psi_{\tau}:H_{\tau}\to C_x$ and $\psi_{\xi}:H_{\xi}\to C_x$, we get smooth maps $\eta\to X_{H_{\tau}},X_{H_{\xi}}$, which are compatible with the immersion into $X$. 
Since $\eta,\tau,\xi$ are connected, we have $\eta$ is mapped into $\tau$ and $\xi$. 
As a consequence, we have $\im (\eta) \subset \im (\tau) \cap \im (\xi)$, i.e.\  $\eta\le \tau$ and
$\eta \leq \xi$.
	
\eqref{le_3} The only possibility for $\tau\to X$ failing to be injective, hence an embedding, is when there are two isomorphic subgroups $K,K'\subset G$ in the local uniformizer $G\ltimes \R^n$ around $x$, such that $K,K'$ are not conjugate, but $\RR^n_K/N_G(K)$ and  $\RR^n_{K'}/N_G(K')$ are in the same component $\tau$ (i.e.\ they are connected outside of the local uniformizer). 
In this case, the connected component of $X_{C_x}$ containing 
$(x,C_x)$
is properly contained in $\tau$.

\eqref{le_4} It is clear that $X$ as a component of $X_{H_{0,X}}$ 
is a maximal element of $\CrS(X)$. 
Moreover, since $\CrS(X)$ is poset, the maximal element is unique. 
\end{proof}

\begin{remark}
    It is important to note the differences between $\tau$ and $X_{\tau}$, the former is an orbifold, while the latter is a more singular object, as $\tau\to X$ is not an embedding in general. However, $\tau \backslash \sing(\tau)\to X$ is an embedding by \eqref{le_3} of Proposition \ref{prop:le_property}.
\end{remark}

\begin{definition}
A stratification of a topological space $X$ by a poset $S$ is an upper semi-continuous map $X\to S$, i.e.,  for any $\mathfrak{s}\in S$, the set $X_{\ge \mathfrak{s}}$ of points of $X$ belonging to a stratum bigger or equal to $\mathfrak{s}$ is an open subset of $X$.
\end{definition}

\begin{proposition}
The natural map $X\to \CrS(X)$ defined by $x\mapsto \min \{\tau|\,x\in X_\tau\}$ is a stratification.
\end{proposition}

\begin{proof}
First, because of \eqref{le_2} in \Cref{prop:le_property} and the fact that $\CrS(X)$ is finite, the assignment $x\mapsto \min \{\tau|\,x\in X_\tau\}$ is well-defined. 
Now given a stratum $\tau$, the set of strata $A=\{\,\xi\;|\;\tau \not \le  \xi \,\}$ is a finite set, which does not contain $\tau_{\max}$. 
Since $X_{\le \xi}=\im (\xi)$ is a closed subset, and $X_{\ge \tau}$ is just the complement of $\bigcup\limits_{\xi\in A} 
X_{\le \xi}$, it follows that $X_{\ge \tau}$ is open, as desired.
\end{proof}

It is clear that $\tau$ is never an effective orbifold, as $H_{\tau}$ always acts trivially in a local chart. Then the reduced orbifold $\tau_{\rm{R}}$ is simply $\tau$ with 
isotropy group the quotient group by $H_{\tau}$, i.e.\ $N_{C_x}(H_{\tau})/H_{\tau}$, by Proposition \ref{prop:C_H_orbifold}. The following observation will be used in the inductive construction of lattices on orbifolds in Section \ref{s:lattice}.

\begin{proposition}\label{prop:sing}
Let $x\in \sing(\tau_{\rm{R}})$. With a bit abuse of language, we also use $x$ to denote the image of $x$ under $\tau\to X$ (note that $|\tau|=|\tau_{\rm{R}}|$) in $X_{\tau}$. Then there exists  $\eta<\tau$, such that $x\in X_{\eta}$.
\end{proposition}

\begin{proof}
If $x\in \sing(\tau_{\rm{R}})$, then $N_{C_x}(H_\tau)\ne H_{\tau}$. Then we can take $\eta$ to be the effective stratum containing $(x,C_x)$, which is clearly smaller than $\tau$. 
\end{proof}

Let $\CrS(X)$ be the set of strata of the orbifold $X$. 
We denote by $\tau_{\max}$ the
main or top stratum. We also define
a height map $h:\CrS(x)\to \N$ by
 \begin{equation}\label{eqn:height}
 h(\tau):=\max\left\{k|\, \exists \tau=\tau_k
 <\ldots<\tau_1<\tau_{\max}\right\}.
 \end{equation}

In the construction of asymptotically holomorphic sections on an orbifold, we will work by induction on the stratification above. 
More precisely, we will find asymptotically holomorphic sections supported in the neighborhoods of each stratum with certain transversality property, by induction on the strata using the order on the set $\CrS(X)$. 
We point out however that $\CrS(\tau)$ is \emph{not} the same as $\CrS(X)_{\le \tau}$ in general,
as the following simple example already shows.

\begin{ex}
Let $D$ denote the orbifold given by the quotient of $\{|x|,|y|\le 1\}\subset \RR^2$ by an action of $\Z/2\times \Z/2$, which acts by reflection on each coordinate. 
We can glue $\Z/2\ltimes (-\epsilon,\epsilon)\times [-1,1]$, where $\Z/2$ acts by reflection on the first coordinate, to $(-\epsilon,\epsilon)\times \{1\}\subset \partial D$ and $\{1\}\times (-\epsilon,\epsilon) \subset \partial D$. 
We denote by $X$ the glued orbifold.
Then, $\CrS(X)$ has three elements $\{\tau_{\max}\ge \tau_1\ge \tau_0\}$, where $\tau_0$ is given by $\bullet/(\Z/2\times \Z/2)$ and $\tau_1$ is given (up to isomorphism) by $[-3,3]$ equipped with a trivial $\Z/2$ action at every point and with an additional $\Z/2$ action by reflection at the points in the subset $\{-2,2\}$. 
Then $\CrS(\tau_1)$ contains three elements $\{\tau_1\ge \tau'_0,\tau''_0\}$ and both $\tau'_0,\tau''_0$ are isomorphic to $\tau_0$. 
In particular, $\CrS(\tau_1)$ is \emph{not} the same as $\CrS(X)_{\le \tau_1}$.
Moreover, this example also shows that in general $\tau \to X$ is \emph{not} an embedding.
\end{ex}	

Although we will not need the following, we include the discussion of ``associativity" of strata in the following.

\begin{proposition}
For any $\tau\in  \CrS(X)$ there is a natural surjective map $\iota_{\tau}:\CrS(\tau)\to \CrS(X)^{\le \tau}$ respecting ordering such that the following properties are satisfied.
\begin{itemize}
    \item $\iota_{\tau}$ maps the unique maximal element $\tau$ of $\CrS(\tau)$ to $\tau\in \CrS(X)$. 
    \item For any $\xi\in \CrS(\tau)$, there is a natural  map $f_{\xi}:\xi\to \iota_\tau(\xi)$, such that the following commutes,
$$
\xymatrix{
\xi \ar[d]^{f_{\xi}} \ar[r] & \tau   \ar[r] & X\ar[d]^{\rm{id}}\\
\iota_{\tau}(\xi)\ar[rr] && X
}
$$
Here $f_{\xi}$ is surjective on orbit spaces and is submersive (i.e.\ $d f_{\xi}$ is surjective). Moreover, $f_{\xi}$ on isotropy groups is always injective (hence $f_{\xi}$ can be viewed as a branched cover). 
\end{itemize}
\end{proposition}

\begin{proof}
Let $\CrC$ be a groupoid representing the orbifold $X$.
Let $\xi\in\mathcal{S}(\tau)$, and $H_\xi$ its associated group. In other words, $\xi$ is the orbit space of a connected component of $(\CrC_{H_{\tau}})_{H_\xi}$.
Now, because of the local charts in \Cref{prop:C_H_orbifold}, the objects of the latter orbifold are just 
triples $(x,K,K')$ with $K\cong H_\tau\subset C_x$ and $K\subsetneq  K'\subset N_{C_x}(K)$ with $H_{\xi}\cong K'$.
The map $\iota_\tau$ then associates to the connected component $\xi$ the connected component $\iota_\tau(\xi)$ in $\mathcal{S}(X)$ which is the orbit set of the component of $\CrC_{H_\tau}$ containing $(x,K')$.

This map has values in the set of connected components which are $\leq \tau$ by construction.
Moreover, it preserves the natural order of the two posets.
Lastly, the desired covering map $f_\xi\colon \xi\to \iota_\tau(\xi)$ for each $\xi \in \mathcal{S}(\tau)$ is simply induced by the map $(x,K,K')\mapsto (x,K')$, where $x$, $K$ and $K'$ are as above. 
It is clear that $f_{\xi}$ is surjective on orbit spaces. Moreover, on the object level, $f_{\xi}$ is the identity map on the fixed subspace of the $K'$-action, hence $d f_{\xi}$ is surjective. While on morphism level, $f_{\xi}$ is the inclusion  $N_{N_{C_x}(K)}(K')\subset N_{C_x}(K')$.
\end{proof}

\section{Geometric structures on orbifolds}\label{s:geo}

Various geometric structures on orbifolds can be defined using the atlas in Definition \ref{def:geo_orbifold} by requiring local invariant or equivariant property and compatibility with change of charts. We list the relevant structures in the following. All orbifolds in the section are assumed to be geometric, i.e.\ effective.
\begin{itemize}
    \item We define {\em orbifold functions} on an orbifold $X$ as orbifold maps $f\colon X \rightarrow\RR$,  considering  $\RR$ as an orbifold. We denote by $C^\infty_{\orb}(X)$ the set of orbifold functions. In other words, a smooth function on $X$ is a continuous function on the topological space $X$, which admits $C^\infty$ invariant lifting to each chart.
    \item An orbifold $X$, with atlas $\{(U_i, {\widetilde U}_i, H_i, \varphi_i)\}$, is {\em oriented} if each $ {\widetilde U}_i$ is oriented, the $H_i$ are subgroups of $\SOO(n)$, and all the changes of charts $\rho_{ij} \colon{\widetilde U}_i\rightarrow{\widetilde U}_j$ are orientation-preserving.
    \item An (orbifold)  {\em Riemannian metric} $g$ on $X$ is a positive definite symmetric tensor in $T^* X\ox T^*X$. This is equivalent to having, for each orbifold chart $(U_i, {\widetilde U}_i, H_i, \varphi_i)$ on $X$, a Riemannian metric $g_{i}$ on the open set ${\widetilde U}_i$ that is invariant under the action of $H_i$ on ${\widetilde U}_i$ ($H_i$ acts on ${\widetilde U}_i$ by isometries), and for which the changes of charts $\rho_{ij}\colon{\widetilde U}_i\rightarrow{\widetilde U}_j$ are isometries, that is $\rho^{*}_{ij}\big(g_{j}\vert_{\rho_{ij}({\widetilde U}_i)}\big)\,=\,g_{i}$.
    \item An (orbifold)  {\em almost complex structure} $J$ on $X$ is an endomorphism $J\colon TX\rightarrow TX$ such that $J^2= -\mathrm{Id}$. Thus, $J$ is determined by an almost complex structure $J_i$ on ${\widetilde U}_i$, for every orbifold chart $(U_i, {\widetilde U}_i, H_i, \varphi_i)$ on $X$, such that the action of $H_i$ on ${\widetilde U}_i$ is by biholomorphic maps, and any change of charts  $\rho_{ij}\colon{\widetilde U}_i\rightarrow{\widetilde U}_j$ is a holomorphic embedding.
    \item An {\em orbifold $p$-form} $\alpha$ on $X$ is a section of $\bigwedge^p T^* X$. This means that, for each orbifold chart  $(U_i, {\widetilde U}_i, H_i, \varphi_i)$ on $X$, we have  a differential $p$-form $\alpha_i$ on the open set ${\widetilde U}_i$, such that every $\alpha_i$ is $H_i$-invariant (i.e. $h^{*}(\alpha_i)= \alpha_i$, for  $h\in H_i$), and any change of charts  $\rho_{ij}\,\colon\,{\widetilde U}_i\,\longrightarrow\,{\widetilde U}_j$ satisfies  $\rho^{*}_{ij}(\alpha_j)=\alpha_i$. The space of $p$-forms on $X$ is denoted by $\Omega_{\orb}^{p}(X)$.
    \item  Observe that an orbifold $X$ of dimension $n$ is oriented if and only if there exists a globally non-zero orbifold form of degree $n$, that is called a \emph{volume form} of $X$. 
    \item The wedge product of orbifold forms and the exterior differential $d$ on $X$ are well defined. The \emph{orbifold de Rham cochain complex} is defined:
    $$
    \cdots \  \overset{d}{\longrightarrow}\Omega_{\orb}^{p}(X) \, \overset{d}{\longrightarrow} \, \Omega_{\orb}^{p+1}(X) \ \overset{d}{\longrightarrow} \ \cdots
    $$
    and its cohomology is the \emph{orbifold de Rham cohomology} of $X$, which is denoted $H_{\orb,\mathrm{dR}}^*(X)$. This is isomorphic to the usual singular cohomology with real coefficients \cite{CFM},
 $$
 H^*_{\orb,\mathrm{dR}}(X)\cong H^*(X,\RR).
 $$
    \item There is another notion of cohomology of orbifolds, which captures more information of the isotropy groups. Namely, we can take the cohomology of the classifying space $BX$ of $X$. We define $H^*_{\orb}(X,\ZZ)=H^*(BX,\ZZ)$ \cite[Section 2.1]{Adem}. By \cite[Proposition 2.11]{Adem}, when using $\RR$-coefficients, we also have $H^*_{\orb}(X,\ZZ)\otimes \R \cong H^*_{\orb,\mathrm{dR}}(X)\cong H^*(X,\RR)$.
\item A \emph{symplectic orbifold} $(X,\omega)$ is an orbifold $X$ equipped with an
orbifold $2$-form $\omega\in \Omega^2_{\orb,\mathrm{dR}}(X)$ such that
$d\omega=0$ and $\omega^n>0$, where $2n=\dim X$. 
In particular, a symplectic orbifold  is oriented.
\end{itemize}

We have a Darboux theorem for symplectic orbifolds \cite[Proposition 11]{MuRo}.

\begin{proposition}\label{DTh}
Let $(X,\o)$ be a symplectic orbifold and $x \in X$. There exists an orbifold chart $(U, V,\phi, H)$ around $x$
with local coordinates $(x_1,y_1,\ldots, x_n,y_n)$ such that the symplectic form
has the expression $\o= \sum dx_i \wedge dy_i$ and $H < \UU(n)$ is a subgroup of the unitary group.
\end{proposition}

\begin{definition}
An \emph{almost K\"ahler orbifold} $(X,J,\omega)$ consists of an orbifold $X$, and orbifold almost complex structure $J$ and
an orbifold symplectic form $\o$ such that $g(u,v)=\o(u,Jv)$ defines an orbifold Riemannian metric with $g(Ju,Jv)=g(u,v)$. Such almost complex structure is called \emph{compatible} (with $\o$).
\\
A \emph{K\"ahler orbifold} is an almost K\"ahler orbifold satisfying the integrability condition that the Nijenhuis tensor $N_J=0$.
This is equivalent to requiring that the changes of charts are biholomorphisms of open sets of $\CC^n$.
\end{definition}

By an identical argument to \cite[Proposition 4.1.1]{MS94}, we have:
\begin{proposition} \label{almost-Kahler-orbifold}
Let $(X,\o)$ be a symplectic orbifold. Then $(X,\o)$ admits an almost K\"ahler orbifold structure $(X, \o, J , {g})$. Moreover, the space of $\o$-compatible almost complex structures is contractible.
\end{proposition}

Note that an almost K\"ahler orbifold structure determines a bigrading of the orbifold $k$-forms. Certainly $TX\otimes \CC=T_{1,0}X\oplus T_{0,1}X$
according to the $(\pm i)$-eigenspaces of $J$ on it. This determines the decomposition on $1$-forms by duality $\Omega^1(X,\CC)=\Omega^{1,0}(X)
\oplus \Omega^{0,1}(X)$. By taking the wedge,
 $$
 \Omega^k(X,\CC)=\bigwedge\nolimits^k \Omega^1(X,\CC)= \bigoplus_{p+q=k} \bigwedge\nolimits^p\Omega^{1,0}(X)
 \ox \bigwedge\nolimits^q \Omega^{0,1}(X) =:  \bigoplus_{p+q=k} \Omega^{p,q}(X).
 $$
This defines projections $\pi_{p,q}:\Omega^{p+q}(X,\CC) \to \Omega^{p,q}(X)$ and so the differential decomposes as
 $$
  d= \pi_{p-1,q+2}\circ d +\pi_{p,q+1}\circ d +\pi_{p+1,q}\circ d +\pi_{p+2,q-1}\circ d
$$
on $\Omega^{p,q}(X)$.
We denote $\bd=\pi_{p+1,q}\circ d \colon \Omega^{p,q}(X) \rightarrow \Omega^{p+1,q}(X)$ and $\bar\bd=\pi_{p,q+1}\circ d\colon \Omega^{p,q}(X) \rightarrow \Omega^{p,q+1}(X)$. On $\Omega^0(X,\CC)$, we have $d=\bd+\bar\bd$.

Now we introduce the concept of (complex) line bundle over an orbifold.

\begin{definition}\label{def:orbibundle}
Let $X$ be an orbifold of dimension $n$, and let $\{(U_i, {\widetilde U}_i, H_i, \varphi_i)\}_{i\,\in\,I}$ be
 an atlas on $X$. An {\em orbifold complex line  bundle} over $X$ consists of a smooth orbifold $L$ of dimension $n+2$, and
an orbifold map $\pi\,\colon\, L\,\to\, X$, called {\em projection}, satisfying:
\begin{enumerate}
\item[i)] For every orbifold chart $(U_i, {\widetilde U}_i, H_i, \varphi_i)$, there
exists a homomorphism $\rho_{i}\,\colon\, H_{i} \,\to\, \UU(1)$ and an orbifold
chart  $(V_i, {\widetilde V}_i, H_i, \Psi_i)$ on $L$, such that
$V_i\,=\,\pi^{-1}(U_i)$, ${\widetilde V}_i={\widetilde U}_i \x \CC$, the action of $H_i$
on ${\widetilde U}_i  \x \CC$ is the diagonal action (i.e.\ $h\cdot (x, u)=(h\cdot x, \rho_{i}(h)(u))$, for
$h\in H_{i}$, $x\in{\widetilde U}_i$ and for $u\in \CC$), and the map
$$
\Psi_i\,\colon\, {\widetilde V}_i={\widetilde U}_i \x \CC\,\to\, L_{|U_{i}}\,:=\, \pi^{-1}(U_i)
$$
is such that $\pi_{|V_i}\,\circ\,\Psi_i\,=\,\varphi_i\,\circ\,{\mathrm {pr}}_{1}$,
where ${\mathrm {pr}}_{1}\,\colon\, {\widetilde U}_i \x \CC\,\to\, {\widetilde U}_i$ is the natural projection,
$\Psi_i$ is $H_i$-invariant
for the action of $H_i$ on ${\widetilde U}_i  \x \CC$, and it induces a homeomorphism
$({\widetilde U_i} \x \CC)/H_i\,\cong\,L_{|U_{i}}$.

\item[ii)] If $(U_i, {\widetilde U}_i, H_i, \varphi_i)$ and $(U_j, {\widetilde U}_j, H_j, \varphi_j)$ are
two orbifold charts on $X$, with $U_i\,\subset\,U_j$, and $\rho_{ij}\,\colon\,{\widetilde U}_i\,\to\,{\widetilde U}_j$
is a change of charts, then there exists a differentiable map, called {\em transition map},
$g_{ij}\,\colon\, {\widetilde U}_i \,\to\, \UU(1)$,
and a change of charts
$\lambda_{ij}\,\colon\,{\widetilde V}_i={\widetilde U}_i \x \CC\,\to\,{\widetilde V}_j={\widetilde U}_j \x \CC$ on $L$, such that
$$
\lambda_{ij}(x, u)\,=\,\big(\rho_{ij}(x), g_{ij}(x)(u)\big),
$$
for all $(x, u)\in{\widetilde U}_i \x \CC$.
\end{enumerate}
\end{definition}

By \cite[Section 2.6]{gironella2021exact}, such orbifold complex bundles are classified by $H^2_{\orb}(X,\ZZ)=H^2(BX,\ZZ)$, using the orbifold first Chern class. A {\em section} (or \emph{orbifold smooth section}) of an orbifold complex line bundle $\pi\colon L\,\to\, X$ is an
orbifold map $s \colon X\,\to\, L$ such that $\pi\,\circ\,s\,=\,1_{X}$. Therefore, if $\{(U_i, {\widetilde U}_i, H_i, \varphi_i)\}$
is an atlas on $X$, then $s$ consists of a family of smooth maps $\{s_{i}\,\colon\,{\widetilde U}_i\,\to\, \CC\}$,
such that every  $s_{i}$ is $H_i$-equivariant and compatible with the changes of charts on $X$.
We denote the space of (orbifold smooth) sections of $L$ by $C^\infty_{\orb}(L)$.

However, if $\rho_i$ is nontrivial, the equivariant condition of a section $s$ forces it to be zero on the fixed point set of $H_i$, which is bad news for finding transverse sections. 
Hence we introduce the following. Let $X$ be a smooth orbifold of dimension $n$ with atlas $\{(U_i, {\widetilde U}_i, H_i, \varphi_i)\}$.
A \emph{complex line bundle} $L\to X$ is an orbifold line bundle $L$ with trivial actions $\rho_{i}$. Note that in this case, the transition maps
 $g_{ij}$ define  maps $g_{ij}:U_i\to \UU(1)$, satisfying the cocycle condition. In particular, $L\to X$ is a topological
line bundle. Using the exponential map for sheaves $\ZZ \to C^\infty_{\orb} \to C^\infty_{\orb}(-,\UU(1))$ and that $C^\infty_{\orb}$ is a flasque sheaf (since it has partitions of unity), we get that
 \begin{equation}\label{eqn:ah}
  H^2(X,\ZZ) \cong H^1(X,C^\infty_{\orb}
  (-,\UU(1))).
  \end{equation}
As the right hand side of (\ref{eqn:ah}) parametrizes the cocyles of orbifold sections, that is, complex line bundles over $X$, we have that the Chern class in $H^2(X,\ZZ)$ classifies the complex line bundle $L\to X$.

We shall denote by $\bar H^2(X,\ZZ)$ the image of $H^2(X,\ZZ)$ in $H^2(X,\RR)=H_{\orb,dR}^2(X)$. This is isomorphic the torsion free part of $H^2(X,\ZZ)$.

\begin{remark}
There is natural map $\pi:BX \to X$ with fiber $\pi^{-1}(x)$ homotopic to $BH_x$, where $H_x$ is the isotropy group at $x$ and $BH_x$ is the classifying space. In particular, we have a natural map $\pi^*:H^2(X,\ZZ)\to H^2_{\orb}(X,\ZZ)$, which corresponds to the obvious lift of a complex line bundle to an orbifold complex line bundle. When $X$ is compact, then there are only finitely many isotropy groups. If $\ell$ is the minimal common multiple of the cardinality of the isotropy groups, then 
$\rho_i$ is trivial for $L^{\otimes \ell}$ for any orbifold complex line bundle $L$. In other words, we have $\ell\cdot H^2_{\orb}(X,\ZZ) \subset \im (\pi^*)$.
\end{remark}

\begin{proposition}
Consider a symplectic orbifold $(X,\o)$. Then 
\begin{enumerate}
  \item There exists a symplectic form  $\widetilde{\omega}$ such that 
  $[\,\widetilde{\omega}/2\pi]\in \bar H^2(X,\ZZ)$, 
  and there is $\ell \in \NN$ with $\widetilde{\omega}/\ell$ that is $C^\infty$-close to $\o$.
  \item There exist a complex line bundle $L\to X$ with $c_1(L)=[\,\widetilde{\omega}/2\pi]$ and  an orbifold connection
$\nabla: C^\infty_{\orb}(L) \to \Omega^1_{\orb}(L)$, with curvature
 $$
 F_\nabla=- i\,\widetilde{\omega}.
 $$
\end{enumerate}
\end{proposition}

\begin{proof}
For (1),  since $[\o]\in H^2_{\orb,\mathrm{dR}}(X)=H^2(X,\RR)$, by density,
we have that a small perturbation $A\in H^2(X,\QQ)$. By
\cite{BBFMT}, Hodge theory holds for orbifolds. Hence there is a $C^\infty$-small closed $2$-form $a\in \Omega^2_{\orb}(X)$
such that $\o'=\o+a$ is closed and $[\,\o'/2\pi]=A 
\in H^2(X,\QQ)$. A small variation in $C^0$-norm of a symplectic form is
again a symplectic form, hence $\o'$ is symplectic. Take a large multiple $\widetilde{\omega}=\ell\, \o'$, $\ell \in \NN$, so that
 $$
  [\,\widetilde{\omega}/2\pi] \in \bar H^2(X,\ZZ).
  $$

For (2), the class $[\,\widetilde{\omega}/2\pi] \in \bar H^2(X,\ZZ)$ lifts
to a class in $H^2(X,\ZZ)$ that determines a complex line bundle $L\to X$ with $c_1(L)=[\,\widetilde{\omega}/2\pi]$. Take charts $(U_i,\tilde U_i,H_i,\phi_i)$ for $X$ and $(V_i,\tilde V_i,H_i,\Psi_i)$ for $L$ with $V_i=\pi^{-1}(U_i)\subset L$ and $\tilde V_i= \tilde U_i\x \CC$. We can take connections $\nabla_i$ on $L\vert_{U_i}$ and   $\{f_i\}$ an orbifold partition of unity subordinated to $\{U_i\}$, as given by \cite[Proposition 5]{MuRo}.
Now
  $$
   \nabla=\sum \pi^*f_i \cdot \nabla_i
   $$
  defines an orbifold connection. Its curvature $F_{\nabla}\in \Omega^2_{\orb}(X)$ has orbifold cohomology class $[F_{\nabla}]= - 2\pi i\, c_1(L)=[- i\,\widetilde{\omega}]$.
  Then there exist $b\in \O^1_{\orb}(X)$ such that
  $F_{\nabla}=-i\,\widetilde{\omega} -db$.
  A new orbifold connection can be defined by 
  $\nabla'= \nabla +b$ and its curvature is
  $F_{\nabla'}=F_{\nabla}+ db= -i\,\widetilde{\omega}$.
\end{proof}

From now on, we shall assume that $(X,\omega)$ is
a symplectic orbifold such that $[\omega/2\pi]\in \bar H^2(X,\ZZ)$. Then $L\to X$ is a complex
line bundle with $c_1(L)=[\omega/2\pi]$. Let $\nabla$ be
a connection with $F_\nabla=- i \omega$,
Associated to a given connection $\nabla$, we have $\bd$ and $\bar\bd$ operators on sections of $L$,
 \begin{align*}
 \bd_L = \pi_{1,0} \circ \nabla: C^\infty_{\orb}(L) \to \Omega^{1,0}_{\orb}(L), \\
 \bar\bd_L = \pi_{0,1} \circ \nabla: C^\infty_{\orb}(L) \to \Omega^{0,1}_{\orb}(L).
 \end{align*}
 
 If $k\in \NN$, the complex line bundle $L^{\otimes k}  \to  X$ has connection $\nabla^k$ with curvature $F_{\nabla^k}= k F_{\nabla}$ and operators 
  \begin{align*}
 \bd_k = \pi_{1,0} \circ \nabla^k: \SC^\infty_{\orb}(L^{\otimes k}) \to \Omega^{1,0}_{\orb}(L^{\otimes k}), \\
 \bar\bd_k = \pi_{0,1} \circ \nabla^k: \SC^\infty_{\orb}(L^{\otimes k}) \to \Omega^{0,1}_{\orb}(L^{\otimes k}).
 \end{align*}

With slight abuse of notation, we shall denote
$\nabla^k,\bd_k,\bar\bd_k$ by $\nabla,\bd,\bar\bd$ again.

\subsection*{Orbifold fundamental group}
Let $X$ be an orbifold. The orbifold fundamental group is defined as 
 $$
 \pi_1^{\orb}(X):=\pi_1(BX),
 $$
see \cite[Definition 4.3.6]{BG}, which is the same as homotopy classes of orbifold maps from $(S^1,*)$ to $(X,x_0)$ with $x_0$ a smooth points, see \cite{Chen06bis}.

The only
case that we shall need is for symplectic orbifolds $(X,\omega)$.
In this case, the isotropy groups are subgroups of
the unitary group $H_x<\UU(n)$ in view of \Cref{DTh}. In particular,
the uniformizers are of the form $H_x \ltimes \C^n$,
and the components $\tau\in \CrS(X)$ are modelled in
complex subspaces $W\subset \C^n$, in particular of
even dimension. 
We denote by $D_i$, for  $i=1,\ldots, r$, the images in $X$ of the connected components in $\mathcal{S}(X)$ (as defined in \Cref{sec:orbifolds}) of the isotropy locus of codimension $2$. 
The associated isotropy groups are always cyclic, say $\ZZ_{m_i}$, because finite subgroups of $\mathrm{U}(1)$ are of this form.
Let $\gamma_i$ be a homotopy class of a loop around $D_i$. All such possibilities are conjugated. Let
$P\subset X$ be the union of the isotropy locus of codimension $\geq 4$. Let $D=\bigcup D_i$. 
Then we have the following equality, that will serve
us to compute the orbifold fundamental group
\begin{equation}\label{def:orb_fundamental}
 \pi_1^{\orb}(X) = \frac{\pi_1(X-(P\cup D))}{\la \gamma_i^{m_i} |\, i =1,\ldots,r\ra} \, .
\end{equation}
Such formula in the context of K\"ahler orbifolds can be found in \cite[Section 2]{Campana}. To see \eqref{def:orb_fundamental}, since every orbifold loop can be arranged to avoid $P\cup D$ up to homotopy, we have a surjection from $\pi_1(X-P\cup D)$ to $\pi_1^{\orb}(X)$. It is easy to see that the kernel is generated by the boundaries of maps of disks to $X$ that intersect $D$, i.e.\ generated by $\gamma_i^{m_i}$.

\section{Lattices in an orbifold}\label{s:lattice}

Let $(X,\omega)$ be a symplectic orbifold.
We fix a compatible almost complex structure $J$, and let $g$ be the associated Riemannian
metric $g=\omega(\cdot,J\cdot)$.
Let $\omega_k=k\,\omega$ be the rescaled symplectic form. 
The associated Riemannian metric
is $g_k=k\, g$. We denote by $d_k$ the 
distance associated to the metric $g_k$.

As it will be needed in the following section in order to find transverse asymptotically holomorphic sections, we now want to find a lattice satisfying properties similar to \cite[Lemmas 12 and 16]{Do96}, at least outside a neighborhood of the singular locus. 
To this end, we will first study the local picture modelled on $H\ltimes \CC^n$ for a finite group $H$ with a homomorphism $\rho:H\to \UU(n)$. Here, we define the singular set 
 $$
 \sing(\rho):=\left\{z\in \CC^n\left| \, \exists h\ne 1\in H, \rho(h)z=z \right.\right\},
 $$
which is a union of complex subspaces.

\begin{definition}\label{def:property_P}
 Let $H$ be a finite group with a homomorphism $\rho:H\to \UU(n)$. 
 We say that $(H,\rho,\CC^n)$ has the \emph{property (P)} if there are constants $C>0$ and $m\in \N$, such that for any $D\gg 0$, there exists a lattice $\Lambda$ on the complement of $N_{CD}(\sing(\rho))$, the neighborhood of radius $CD$ of $\sing(\rho)$, with the following properties.
 \begin{enumerate}
     \item (Covering property) The balls of radius $1$ around $\Lambda$ cover  $\CC^n\backslash N_{CD}(\sing(\rho))$.
     \item (Even distribution) For $q\in \CC^n$, we define
     $$
     F_q(s)=\#\{B_s(q)\cap \Lambda \},
     $$
     where $B_s(q)$ is the ball of radius $s$ centered at $q$. Then, $F_q(s)<C s^{2n}$.
     \item (Strong $D$-separation) $\Lambda$ has a partition into $N=CD^{m}$ families $\Gamma_1,\ldots,\Gamma_N$  such that if $x\ne y\in \Gamma_i$, then $d(x,\rho(h)y)\ge D$ for any $h\in H$, and $d(x,\rho(h)x)\ge D$ if $h\ne 1$ and $x\in \Gamma_i$.
 \end{enumerate}
\end{definition}

We start with the one-dimensional case.

\begin{lemma}\label{lemma:1D}
$(H,\rho,\CC)$ has property (P).
\end{lemma}

\begin{proof}
If $\rho:H\to \UU(1)$ is not injective, then $\sing(\rho)=\C$. Hence the claim is tautological. Therefore we can assume $\rho$ is injective. It then follows also that $H=\ZZ/k\ZZ$ with $\rho(1)=e^{{2\pi i}/{k}}$.

We start with the lattice $\Lambda$ of integer points. 
For $D\in \N$, $\Lambda$ can be partitioned into $D^2$ families $\Gamma_{i,j}$ for $0\le i <D$ and $0\le j <D$, where $(x,y)\in \Gamma_{i,j}$ if and only if $x \equiv i, y\equiv j \pmod D$. 
It is clear that the covering and even distribution properties hold for $\Lambda$. 
We then claim that $(H,\rho,\CC)$ has property (P) for $C=\max\{1/|e^{{\pi i}/{k}}-1|,1/|e^{{2\pi i}/{k}}-1|,2k\}$ and $m=2$. 
More precisely, the desired lattice is given by $\Lambda\cap (\C \backslash B_{CD}(0))$. 

To construct the desired partition, we first divide  $\C \backslash B_{CD}(0)$ into $2k$ angular sectors $S_1,\ldots,S_{2k}$, each with angle $\frac{\pi}{k}$. 
Then we partition $\Gamma_{i,j}$ further into $\bigsqcup\limits_{r=1}^{2k}\Gamma_{i,j,r}$, where $\Gamma_{i,j,r}:=\Gamma_{i,j}\cap S_r$. 
For $x\ne y \in \Gamma_{i,j,r}$ and $h=l\in \ZZ/k\ZZ$, we have $\rho(h)y\in S_{r+2l}$. 
Now, if $l\ne 0$, then the distance between $S_r,S_{r+2l}$ is at least $CD|e^{{\pi i}/{k}}-1|\ge D$, so that $d(x,\rho(h)y)>D$ in particular. 
If $l=0$, then $d(x,\rho(h)y)=d(x,y)>D$ by construction. 
Lastly, for $x\in \Gamma_{i,j,r}$ and $h\ne 0$, we have $|\rho(h)x-x|\ge|x|\cdot|e^{{2\pi i}/{k}}-1|\ge D$. 
This proves that the partition has the strong $D$-separation property as well, thus concluding the proof.
\end{proof}

\begin{lemma}\label{lemma:sum}
If $(H,\rho_1,V_1)$ and $(H,\rho_2,V_2)$ have property (P), then $(H,\rho_1\oplus\rho_2,V_1\oplus V_2)$ has property (P) as well.
\end{lemma}

\begin{proof}
First note that $\sing(\rho_1\oplus \rho_2)=\sing(\rho_1)\times \sing(\rho_2)$. 
By assumption we can find lattices $\Lambda_1,\Lambda_2$ on $V_1\backslash N_{CD}(\sing(\rho_1)),V_2\backslash N_{CD}(\sing(\rho_2))$, respectively, 
with all three properties. 
Then the product lattice $\Lambda_1\times \Lambda_2$ and the associated product partition can cover  $(V_1\backslash N_{CD}(\sing(\rho_1)))\times (V_2\backslash N_{CD}(\sing(\rho_2)))$ and has the strong $D$-separating property. 
(Of course, the balls of radius $1$ in $V_1\oplus V_2$ only cover the product of  balls of radius $\frac{1}{\sqrt{2}}$ in $V_1,V_2$, but we will neglect such discrepancy, which can be accounted for simply by modifying the constant $C$ in property (P).) 

Now, for $N_{CD}(\sing(\rho_1))\subset V_1$, we can apply \cite[Lemmas 12 and 16]{Do96} (forgetting the group action) to get a lattice $\Lambda^S_1$ with the following properties.
\begin{enumerate}
    \item  The balls of radius $1$ around $\Lambda_1^S$ cover 
    $N_{CD}(\sing(\rho_1))$.
    \item  For $q\in V_1$, we have
     $$
      F_q(s)=\#\{B_s(q)\cap \Lambda^S_1\}<Cs^{\,\dimR V_1}.
      $$
     \item $\Lambda^S_1$ has a partition into $CD^{\,\dimR V_1}$ families $\Gamma^{S,1}_1,\Gamma^{S,1}_2,\ldots$, such that if $x\ne y \in \Gamma^{S,1}_i$, we have $d(x,y)>D$.
\end{enumerate}
Then the product lattice $\Lambda^{S}_1\times \Lambda_2$ covers $N_{CD}(\sing(\rho_1))\times (V_2\backslash N_{CD}(\sing(\rho_2)))$. 
We also claim that the product partition has the strong $D$-separation property. 

Indeed, consider $(x_1,x_2),(y_1,y_2)$ in the same subset of the partition.
In the case where $x_2\ne y_2$, then it just follows from the strong $D$-separation property of $\Lambda_2$. 
If $x_2=y_2$ and $x_1\ne y_1$, then $d((x_1,x_2),(y_1,x_2))\ge d(x_1,y_1)\ge D$ and $d((x_1,x_2),(\rho_1(h)y_1,\rho_2(h)x_2))\ge d(x_2,\rho_2(h)x_2)\ge D$ when $h\ne 1$.
Lastly, if $x_1=y_1 $ and $ x_2=y_2$, then $d((x_1,x_2),(\rho_1(h)x_1,\rho_2(h)x_2))\ge d(x_2,\rho_2(h) x_2)\ge D$ for $h\ne 1$. 

Similarly, we can find a lattice for $(V_1\backslash N_{CD}(\sing(\rho_1)))\times N_{CD}(\sing(\rho_2))$ with the covering and strong $D$-separation properties. 
We then claim that the disjoint union 
 $$
 (\Lambda_1\times \Lambda_2) \cup (\Lambda^{S}_1\times \Lambda_2)\cup (\Lambda_1\times \Lambda^S_2),
 $$
with the induced disjoint union of partitions, yields property (P) for   $(H,\rho_1\oplus\rho_2,V_1\oplus V_2)$. The covering property and strong $D$-separation property have been already proved. 

To see that the even distribution property holds, it suffices to show that if it holds for the lattices $\Lambda,\Lambda'$ on $V_1,V_2$ respectively, then it holds for the product lattice $\Lambda\times \Lambda'$. 
For $q=(q_1,q_2)\in V_1\oplus V_2$, note that 
$$B_R(q)\cap (\Lambda \times \Lambda')\subset (B_R(q_1)\cap \Lambda)\times (B_R(q_2)\cap \Lambda').$$
Therefore the claim follows. 
More precisely, the constant $m$ for $(H,\rho_1\oplus\rho_2,V_1\oplus V_2)$ is $\max\{m_1,m_2\}$, where $m_1,m_2$ are the constants for $(H,\rho_1,V_1),(H,\rho_2,V_2)$ in the property (P).
\end{proof}

\begin{lemma}\label{lemma:union}
Assume $H$ is covered by subgroups $H_1,\ldots,H_k$. We write $\rho_i=\rho|_{H_i}$. If $(H_i,\rho_i,V)$ has the property (P) for all $1\le i \le k$, then $(H,\rho,V)$ also has the property (P).
\end{lemma}

\begin{proof}
We prove the case of $k=2$, as the general case is similar. In this case, we have $\sing(\rho)=\sing(\rho_1)\cup \sing(\rho_2)$. Again by assumption, we can find lattices $\Lambda_1,\Lambda_2$ for $V\backslash N_{CD}(\sing(\rho_1))$ and $V\backslash N_{CD}(\sing(\rho_2))$. Let $\{\Gamma^1_1,\ldots,\Gamma^1_i,\ldots\}$ and $\{\Gamma^2_1,\ldots,\Gamma^2_N\}$ be the associated partitions. Our lattice for $(H,\rho,V)$ will be $\Lambda_1\cap  (V\backslash N_{CD}(\sing(\rho_2)))$. It clearly satisfies the covering and the even distribution properties. To see the strong $D$-separation property, we need to refine the partition $\{\Gamma^1_1,\ldots,\Gamma^1_i,\ldots\}$ with respect to $\{\Gamma^2_1,\ldots,\Gamma^2_N\}$ as follows. Each $\Gamma^1_i\cap (V\backslash N_{CD}(\sing(\rho_2)))$ is refined to $\bigsqcup\limits_{j=1}^N \Gamma^1_{i,j}$, where
 $$
 \Gamma^1_{i,j}=\left\{ x\in \Gamma^1_i| \, \exists y \in \Gamma^2_j, d(x,y)<1, \not\hspace{-1pt}\exists z\in \Gamma^2_{k},k<j, d(x,z)<1\right\}.
 $$
The fact that $\bigsqcup\limits_{j=1}^N \Gamma^1_{i,j}$ is a partition of  $\Gamma^1_i\cap (V\backslash N_{CD}(\sing(\rho_2)))$ follows from the covering property of $\Lambda_2$. 
We claim that $\Gamma^1_{i,j}$ has the strong $(D-2)$-separation property for $H=H_1\cup H_2$. Given $x\ne y \in \Gamma^1_{i,j}$, for $h\in H_1$, we have $d(x,\rho_1(h)y)\ge D$ by assumption. If $h\in H_2$, since there are $x',y'\in \Gamma^2_{j}$ with $d(x,x'),d(y,y')<1$, we have 
 $$
 d(x,\rho_2(h)y)\ge d(x',\rho_2(h)y')-d(x,x')-d(\rho_2(h)y,\rho_2(h)y')\ge D-2.
 $$
Similarly, we have 
$d(x,\rho_1(h)x)\ge D$ for $h\ne 1 \in H_1$ and $d(x,\rho_2(h)x) \ge D-2$ for $h\ne 1\in H_2$. We can then rescale $C$ to conclude that $(H,\rho,V)$ has property (P) for the $m$-constant equals to $m_1+m_2$, where $m_i$ are the constants of $(H_i,\rho_i,V)$.
\end{proof}

Now we move to a local chart in the orbifold $(X,\omega)$.

\begin{lemma}
\label{lemma:lattice_in_loc_uniform}
Let $H\ltimes B$ be a local chart for a finite $H\subset \UU(n)$ and $R>0$ be a fixed constant. 
Then there is a constant $C$, depending only on $H$ and $R$, such that the following property holds.
For any $k\gg 1$, on the set $W_{k,CD}$ of points in $W=B/H$ at $g_k$-distance at least $CD$ from the isotropy locus, one can find a finite set of points $\Lambda$ such that:
\begin{enumerate}
     \item\label{lattice_1} The balls of $g_k$-radius $R$ centered at  points of $\Lambda$ cover $W_{k,CD}$.
     \item\label{lattice_2} For $q\in B/H$, we have
     $$
     \sum_{p_i\in \Lambda}d_k(p_i,q)^r e^{-d_k(p_i,q)^2/5}\le C, \qquad r=0,1,2,3.$$
    \item\label{lattice_3}  $\Lambda$ can be divided into $\Gamma_1,\ldots,\Gamma_N$, where $N=O(D^{m})$ (independent of $k$), for some $m$ that only depends on $H$, such that 
    \[
    \forall \alpha=1,\ldots, N, \quad
    \forall p_i,p_j\in \Gamma_{\alpha}, \quad
    d_k(p_i,p_j)\ge D.
    \]
    \end{enumerate}
\end{lemma}

\begin{proof}
One can choose local charts so that $d (x,y)/A\leq d_1(x,y) \leq A d(x,y)$ for some $A>0$, where $d_1$ is the distance coming from the norm $g=\omega(\cdot, J \cdot)$ on the orbifold $X$ and $d$ is the euclidean distance in the local chart $H\ltimes B$.
Hence, up to rescaling by the factor $k$ the euclidean ball, it suffices to the prove the existence of such lattice on $H\ltimes \CC^n$ with the standard metric. 
By rescaling further, we can assume that $R=1$, which affects the universal constant $C$. 

We first claim that $(H,\iota,\CC^n)$ has property (P), where $\iota:H\to \UU(n)$ is the inclusion. First note that $H$ is covered by finitely many $H_i\cong \ZZ/m_i\ZZ$. 
On the other hand, any $\ZZ/m_i\ZZ$ representation can be decomposed into one-dimensional representations.
Therefore $(H,\iota,\CC^n)$ has property (P) by Lemmas \ref{lemma:1D}, \ref{lemma:sum} and \ref{lemma:union}.  

The lattice on the quotient is then the image of the lattice on $(H,\iota,\CC^n)$. 
It is clear that the covering property implies \eqref{lattice_1}. 
For $q\in B/H$, there are at most $|H|$ preimages in $B$, then the even distribution property for those preimages implies \eqref{lattice_2}, with the same estimates as in \cite{Do96}. 
Lastly, because of the strong $D$-separating property, if $x\ne y \in \Gamma_i$, we have $d(x,y)>D$ in $B/H$, hence \eqref{lattice_3} holds.
\end{proof}

Finally, we can assemble lattices from the local charts and strata of $X$. For an orbifold $X$ and a stratum $\tau\in \CrS(X)$, we denote by $X_{\tau,k,CD}$ the complement of a neighborhood of $g_k$-radius $CD$ of $\bigcup\limits_{\theta<\tau} X_\theta$ in $X_{\tau}$.

\begin{proposition}\label{prop:lattice}
Let $R>0$ and $\tau\in\CrS(X)$. Then, there is a universal constant $C$ and $m\in \NN$, depending only on $X$ and $R$ (independent of $k$), such that the following property is satisfied.
For any $D\gg 0$, $k\gg 0$, there is a set of points $\Lambda$ such that:
 \begin{enumerate}
     \item The balls of $g_k$-radius $R$ centered at points of $\Lambda$ cover $X_{\tau,k,CD}$,
     \item\label{p2:bounded} For $q\in X$, 
     \[
     \sum_{p_i\in \Lambda} d_k(p_i,q)^r e^{-d_k(p_i,q)^2/5}\le C, \qquad r=0,1,2,3.
     \]
    \item $\Lambda$ can be partitioned into $\Gamma_1,\ldots,\Gamma_N$, with $N=O(D^{m})$ (independent of $k$), such that
    \[
    \forall \alpha =1,\ldots,N, \,
    \quad
    d_k(p_i,p_j)\ge D, \text{ for } \, \, p_i,p_j\in \Gamma_{\alpha}.
    \]
 \end{enumerate}
\end{proposition}

\begin{proof}
Without loss of generality we assume that $X$ is effective (that is, a geometric orbifold), for otherwise we can take the reduced orbifold $X_{\rm{R}}$. When $\tau=\tau_{\max}$, the statement is an immediate consequence of \Cref{lemma:lattice_in_loc_uniform}, as lower strata are precisely $\sing(X)$. 

In general, let $\iota$ denote the immersion $\tau \to X$. 
Applying the same covering argument to the effective orbifold $\tau_{\rm{R}}$, we get a lattice on $|\tau_{\rm{R}}|\backslash N_{CD}(\sing(\tau_{\rm{R}}))$ with the three properties. By Proposition \ref{prop:sing}, we have $\iota^{-1}(\bigcup_{\theta<\tau} X_{\theta})\supset \sing(\tau_{\rm{R}})$. By Proposition \ref{prop:le_property}, we have $\iota|_{\iota^{-1}(X_{\tau,k,CD})}$ is an embedding, hence the metric on  $\iota|_{\iota^{-1}(X_{\tau,k,CD})}$ and $X_{\tau,k,CD}\subset X$ are comparable. In particular, the intersection of $X_{\tau,k,CD}$ with the pushforward of the lattice on  $|\tau_{\rm{R}}|\backslash N_{CD}(\sing(\tau_{\rm{R}}))$ also satisfies the three properties (with a different $C$, independent of $k$).
\end{proof}

\begin{remark}\label{rmk:even_distribution}
As in \cite{Do96}, the even distribution property of \Cref{def:property_P} in fact implies that for any $N\in \NN$, there exists a universal (independent of $k$) constant $C_N$, such that for any $q\in X$, we have  
$$ \sum_{p_i\in \Lambda}d_k(p_i,q)^r e^{-d_k(p_i,q)^2/5}\le C_N, \qquad 0\le r \le N,$$
where $\Lambda$ is the lattice in \Cref{prop:lattice}.
\end{remark}

\section{Asymptotically holomorphic sections}\label{sec:donaldson}

Let $(X,\omega)$ be a symplectic orbifold with $[\omega/2\pi ]\in H^2(X,\ZZ)$.
We fix a compatible almost complex structure $J$, and let $g$ be the associated Riemannian
metric $g=\omega(\cdot,J\cdot)$.
Consider now a Hermitian complex line bundle $L\to X$ with $c_1(L)=[\omega/2\pi]$, and a connection $\nabla$ on it, with curvature $F_{\nabla}= -i\omega$. 
Let also $L^{\otimes k}$, for $k\geq 1$, which has an induced connection, again denoted
by $\nabla$ with a little abuse of notation, whose curvature is $F_{\nabla} =-i \omega_k=-ik\,\omega$, where
$\omega_k=k\,\omega$ is the rescaled symplectic form. 
The associated Riemannian metric
is then just $g_k=k\, g$. We denote by $d_k$ the 
distance associated to the metric $g_k$.

Following Donaldson's work \cite{Do96}, we will search for the following objects:

\begin{definition}
\label{asympt_hol_seq}
A sequence of sections $s_k$ of $L^{\otimes k}\to X$ is called \emph{asymptotically $J$-holomorphic} if there exists 
a constant $C>0$ such that 
$|s_k|\leq C$, $|\nabla s_k|\leq C$,
$|\bar{\partial} s_k|\leq C k^{-1/2}$ and
$|\nabla \bar{\partial} s_k|\leq C k^{-1/2}$. 
Here (and in everything that follows), all the norms are evaluated with respect to the metrics $g_k$.
\end{definition}

A transversality condition is needed in order to ensure that the zero sets of the sections are symplectic suborbifolds for $k$ large enough.

\begin{definition} \label{skd-trans}
Given $\eta>0$, a sequence of sections $s_k$ of the line bundle $L^{\otimes k}$ is said to be
\emph{$\eta$-transverse to $0$} if for every point $x\in M$ such that
$|s_k(x)|<\eta$ then $|\nabla s_k(x)|>\eta$.
\end{definition}

\begin{proposition}
\label{prop:sympl_suborb_from_sequence}
Let $s_k$ be an asymptotically $J$-holomorphic sequence of
sections of $L^{\otimes k}$ which are
$\eta$-transverse to $0$, for some $\eta>0$. 
Then
for $k$ large enough, the zero sets $Z(s_k)$ have the structure of
symplectic suborbifolds of $X$.
\end{proposition}

\begin{proof}
We have $|\partial s_k(x)|>|\bar{\partial} s_k(x)|$ if $x$ is a zero
of $s_k$, for $k$ large enough. Suppose that $x\in Z_k:= Z(s_k)$, take an
orbifold Darboux chart $(\tilde U, U, H, \varphi)$. Then 
$s_k$ defines a map $\tilde s_k:\tilde U\to \CC$. As $d\tilde s_k$ is surjective,
the zero set $\tilde V:=Z(\tilde s_k) \subset \tilde U$ is a submanifold. 
As $\tilde s_k$ is $H$-invariant, $H$ acts on $\tilde V$. 
Moreover $V:=Z(s_k)\subset U$ comes with a natural homeomorphism
$\tilde V/H\cong V$. So $(\tilde V,V,H,\varphi|_{\tilde V})$
is a chart for $Z_k$. Next $T_x \tilde V=\ker d\tilde s_k$. As 
$|\partial \tilde s_k(x)|>|\bar{\partial} \tilde s_k(x)|$, this is a symplectic
subspace of $T_x \tilde U=\RR^{2n}$. So $\omega|_{Z(s_k)}$ is an orbifold
$2$-form which is moreover symplectic.
\end{proof}

To find asymptotically holomorphic sections, we need to develop some tools.
We start by a refined Darboux coordinates.

\begin{lemma}\label{lem:trivializa}
 Near any point $x\in X$, for any integer $k\geq 1$, there exist local
 complex Darboux coordinates $(V_k, \tilde V_k,\Gamma,\Phi_k)$ around $x$,
 $\Phi_k=(z_k^1, \ldots, z_k^n): \tilde V_k
 \to (\C^n, 0)$ for the symplectic structure $k\omega$ such that the
 following bounds hold universally:
 \begin{itemize}
     \item $|\Phi_k(y)|^2=O(d_{k}(x,y)^2)$
 on a ball $B_{g_k}(x,c\,k^{1/2})$. 
 \item $|\nabla \Phi_k^{-1}|_{g_k}= O(1)$
 on a ball
 $B(0,c\,k^{1/2})$.
 \item With
 respect to the almost-complex structure $J$ on $X$ and the
 canonical complex structure $J_0$ on $\C^n$,
 $|\bar{\partial} \Phi_k^{-1}(z)|_{g_k}= O(k^{-1/2}|z|)$ and 
 $|\nabla \bar{\partial} \Phi_k^{-1}|_{g_k}=O(k^{-1/2})$ on $B(0,c)$.
 \end{itemize}
\end{lemma}

\begin{proof}
This is  \cite[Lemma 3]{Au99} in the manifold case.
We start with a Darboux coordinate $(U,\tilde U,H,\varphi)$,
where $\varphi:\tilde U \to \C^n$ and 
$H< \UU(n)$.
We modify $\varphi$ as in \cite{Au99},
to get $\Phi_k$ with the stated bounds. The key fact is that
all changes are $\UU(n)$-equivariant, hence they are $H$-invariant.
\end{proof}

The starting point for Donaldson's construction is the following existence
Lemma.

\begin{lemma}\label{lemma:localsection}
There exists a universal constant $C$ such that for $x\in X_{\tau}$ that is at $g_k$-distance at least $CR$ 
from $X_{\theta}$ for any $\theta < \tau$, we can find for any $k\gg 0$ an asymptotically holomorphic section $s_{k,x}$ of $L^{\otimes k}$ with
\begin{enumerate}
    \item $|s_{k,x}|>\frac{1}{2}\exp(-R^2)$ on a ball of $g_k$-radius $R$ centered at $x$, 
    \item $|s_{k,x}(y)|<C e^{-{d_k(x,y)^2}/{5}}$ and $|\nabla s_{k,x}(y)|<C (1+d_k(x,y)) e^{-{d_k(x,y)^2}/{5}}$. 
\end{enumerate}
\end{lemma}

\begin{proof}
We will use the Gaussian section, multiplied by a cut-off bump function supported 
in a ball of $g_k$-radius $k^{1/6}$ around a given point, constructed by \cite{Do96} in the manifold case; we will denote such section by $\tilde{s}_{k,x}$. 
Recall that, for $k$ large, one has $|\tilde s_{k,p}|>\frac{3}{4}\exp(-R^2)$ on a ball of $g_k$-radius $R$ centered at any point $p$ of the manifold. 

Let us now go to our orbifold setting. We start by covering the orbifold $X$ by a finite family of local uniformizers $\{H_i\ltimes U_i\}$ as in \Cref{lem:trivializa}, such that for any $y\in X$, the ball of $g_k$-radius $k^{1/6}$ is contained in one of the uniformizers (this is clearly possible provided that $k\gg 0$). So we can assume that each $y$ is contained in a local uniformizer $H\ltimes B(1)$ with $\mathrm{supp} (\tilde{s}_{k,y})\subset B(1)$.

Let now $x$ denote a point in the stratum $X_\tau$ which has $g_k-$distance at least $CR$ from $X_\theta$ for every $\theta<\tau$.
Since every $h^*\tilde{s}_{k,x}$ for $h\in H$ is supported in $B(1)$,  we have an equivariant section (i.e.\ an orbifold section) by
  $$
  s_{k,x}=\frac{1}{|H_x|} \sum_{h\in H} h^*\tilde{s}_{k,x},
 $$
where $H_x\subset H$ is the subgroup fixing $x$. Since the $H_x$-action preserves $\tilde{s}_{k,x}$, we have
 $$
 s_{k,x}=\sum_{h\in H/H_x} h^*\tilde{s}_{k,x},
 $$
where the sum is over any representative of the left coset $H/H_x$. If $H_x=H$, then the claim follows automatically. 

Suppose that there is $h\in H$, such that $hx\ne x$. Let $d=|H|$. Since $h^d=1$, after a unitary change of coordinate, 
we have $h=\mathrm{diag}(e^{2\pi i m_1/d} ,\ldots,e^{2\pi i m_n/d})$ for $0\le m_1\le \ldots \le m_n<d$. 
We assume $m_l=0$ and $m_{l+1}> 0$ and the fixed space of $h$ is $V=\CC^l\times \{0\}^{n-l}$. 
Write $x=(x_1,\ldots, x_n)$. 
Since $V$ belongs to $X_\theta$ for some $\theta<\tau$,  
$CR\leq d_k(x,X_{\theta})\leq d_k(x,V)$. So
there is some $j\geq l+1$ with $|x_j| \geq \frac{1}{\sqrt n}CR$. Then 
  $$
 |h x-x | \geq |e^{2\pi i m_j/d} x_j- x_j | = |e^{2\pi i/d} - 1|\, |x_j| \geq \epsilon\, CR ,
 $$
for some universal $\epsilon>0$ depending only on $n,d$ (that is $X$ and $H$). 
Now by compactness of $X$, we have finitely many local uniformizers covering $X$, and hence a universal $\epsilon>0$
satisfying that $d_k(hx,x)\geq \epsilon\, CR$. By the exponential decay,
 $$
 |h^*\tilde{s}_{k,x}(y)| \leq C'e^{- d_k(y,hx)^2/5},
 $$
and on the ball $B(x,R)$ one has $d_k(y,hx)\geq (\epsilon C-1)R$, so that 
 $|h^*\tilde{s}_{k,x}| \leq C'e^{- (\epsilon C-1)^2R^2/5}$. For 
$C>0$ large enough, this is smaller than $\frac{1}{4|H|} e^{-R^2}$. Hence
 $$
 |s_{k,x}| \geq |\tilde{s}_{k,x}| - \sum_{h\neq 1} |h^*\tilde{s}_{k,x}| \geq \frac34 e^{-R^2}- \frac14 e^{-R^2}=\frac12 e^{-R^2}\, .
 $$
 
Property (2) is also clear, as $s_{k,x}$ is a finite combination of $\tilde{s}_{k,x}$ such that each one of them satisfies the desired inequalities.
\end{proof}

\begin{remark}\label{rmk:higher}
By a direct computation, the higher derivatives of $s_{k,x}$ will satisfy
$$
|\nabla^p s_{k,x}|<P(d_k(x,y))e^{-d_k(x,y)^2/5}
$$
where $P$ is a universal polynomial of degree $p$,  which does not depend on $k$ and $x$.
\end{remark}

By writing locally near a point $x\in X$ a given section $s$ of $L^{\otimes k}$ as $f\, s_{k,x}$ for some  function $f$ defined on a local orbifold chart with values in $\CC$, one can conveniently rephrase \Cref{skd-trans} in terms of the function $f$ as follows:
\begin{definition}
 A function $f:\C^n \to \C$ is \emph{$\s$-transverse to $w\in \C$ at a point
 $x\in \C^n$} if the inequality  $|f(x)-w|< \sigma$ implies
 $|df(x)|>\sigma$.
\end{definition}

We will use the following rescaled version of \cite[Theorem 20]{Do96}, which is simply deduced from the latter by considering ${f}(z)=\tilde f(\frac{z}{R})$.

\begin{theorem}[Rescaled {\cite[Theorem 20]{Do96}}]\label{thm:transverse}
For $\sigma>0$, let $\mathcal{H}_{\sigma,R}$ denote the functions $f$ on the closed polydisk
$\Delta^+(R):=\{z\in\CC^n\vert \, \vert z_i \vert \leq 11 R / 10\}$, such that 
\begin{enumerate}
    \item $|f|_{C^0(\Delta^+(R))}\le 1$.
    \item $|\overline{\partial} f|_{C^1(\Delta^+(R))} \le \frac{\sigma}{R^2}$.
\end{enumerate}
Then there is an integer depending only on the complex dimension $n$, such that for any $0<\delta<\frac{1}{2}$, if $\sigma < Q_p(\delta)\delta$, then for any $f\in \mathcal{H}_{\sigma,R}$, there is $w\in \CC$ with $|\omega|\le \delta$ such that $f$ is $Q_p(\delta)\delta/R$-transverse to $w$ over the interior $\Delta(R):= \{z\in\CC^n\vert \, \vert z_i \vert \leq  R\}$ of $\Delta^+(R)$. 
\end{theorem}

The approach we are now going to use to find asymptotically holomorphic sections with certain quantitative transversality is to apply the techniques in \cite{Do96} inductively on the strata starting from the \emph{top} stratum. 
For this, we first need the following result.

\begin{lemma} \label{lem:xalpha}
 Let $C_0>0$, $R>0$ and $p>0$. Define $Q_p(\eta)=(\log(\eta^{-1}))^{-p}$.  Starting with 
 $\eta_0>0$, define a sequence by $\eta_i=\eta_{i-1}Q_p(\eta_{i-1})/2R$. Then 
 there is some constant $C>0$ 
 depending on $C_0,R,p,\eta_0$ such that
 $Q_p(\eta_{C_0D^{m}})\geq C/D^{m p+1}$,
 for $D$ large enough.
\end{lemma}

\begin{proof}
 As in \cite[Lemma 24]{Do96}, define 
 $x_\alpha=-\log \eta_\alpha$. Therefore
  $$
  x_\alpha=x_{\alpha-1} + p\log x_{\alpha-1} + \log (2R).
  $$
 Take $q>p$ and introduce $y_\alpha=q\alpha\log \alpha$. Then
 the same proof as that 
 of \cite[Lemma 24]{Do96} shows
 that $y_\alpha-y_{\alpha-1} \geq p \log y_{\alpha+1}+\log (2R)$, for $\alpha$ large enough. So $x_\alpha \leq q(\alpha+\alpha_1)\log(\alpha+\alpha_1)$, for some $\alpha_1$, and hence 
  $$
  Q_p(\eta_{\alpha-1}) \geq \frac{C}{(\alpha\log\alpha)^p}\, .
  $$
This implies in turn that 
  $$
  Q_p(\eta_{\alpha}) \geq \frac{C}{D^{mp+1}},
  $$
for $\alpha\leq C_0\, D^{m}$.
\end{proof}

\begin{proposition}\label{prop:size}
For each stratum $\tau\in \CrS(X)$, there exist positive numbers $D_{\tau}, R_{\tau},C_\tau, \{\eta_{\tau,i} \}_{i\le C_\tau D_{\tau}^{m}}$ of the following significance for $k\gg 0$.
\begin{enumerate}
    \item For $R_\tau$, \Cref{prop:lattice} can be applied to $X_{\tau,k,C_\tau D_{\tau}}$.
    \item If $\theta<\tau$, then 
    $R_\theta>2C_\tau D_{\tau}$, where $C_\tau$ is a constant so that  \Cref{lemma:localsection} and 
    \Cref{prop:lattice} can
    be applied. 
    \item We have $\eta_{\tau,i}=Q_p(\eta_{\tau,i-1})\eta_{\tau, i-1}/2R_{\tau}$, for $i=1,\ldots, C_\tau D_\tau^m$. 
    \item $Q_p(\eta_{\tau,C_\tau D_{\tau}^m})>\exp(-D^2_{\tau})$. 
    \item If $\theta<\tau$, then $\eta_{\theta,1}<
    \frac{1}{2}\eta_{\tau,C_\tau D_{\tau}^m}$.
\end{enumerate}
\end{proposition}

\begin{proof}
This is proved by induction on the height
\eqref{eqn:height} on $\CrS(X)$. 
We start with $\tau_{0}=\tau_{\max}$, where we take $R_{\tau_{0}}=1$ and $\eta_{\tau_{0},0}=\frac{1}{2}$. 
Then there is a universal constant $C_{\tau_0}\gg 0$, such that \Cref{prop:lattice} can be applied to $X_{\tau_0,k, C_{\tau_0}D}$, 
with $D$ to be chosen large enough shortly.
We define $\eta_{\tau_0,i}$ as in
the statement, hence  \Cref{lem:xalpha}
gives the desired lower bound in (4).
Therefore there exists a large enough $D=D_{\tau_{0}}$, such that $Q_p(\eta_{\tau_{0},C_{\tau_0}D^m}) > \exp(-D_{\tau_0}^2)$. 

Now assume the claim holds for all strata with height $\le \ell$. 
Let $\tau_1,\ldots,\tau_s$ be the
strata with height $\ell+1$. 
First note that by definition, we cannot have $\tau_i<\tau_j$ for some $1\leq i,j\leq s$. Take $R$ so
that the neighborhood of $\bigcup \tau_i$ of $g_k$-radius $\frac12 R$ covers the complement of the domains where we applied \Cref{prop:lattice} in the previous steps, whose $g_k-$radius is given by the maximum of the $C_\vartheta D_{\vartheta}$'s
for all $\vartheta$ that contains any of the $\tau_1,\ldots,\tau_s$. Set $R_{\tau_i}=R$ for all $i$. Then we can apply \Cref{prop:lattice} to $X_{\tau_i,k,C_{\tau_i}D_{\tau_i}}$ with this $R$, since this would guarantee the balls of $g_k$-radius $R_{\tau_i}$ at the found lattice would cover the neighborhood of $\bigcup \tau_i$ of $g_k$-radius  $\frac{1}{2}R_{\tau_i}$ (minus the neighborhood of deeper strata), i.e.\ the domain that has been missed in the previous steps. 

Now we take    
 $$
 \eta_{\tau_i,1}:=\min\left\{\frac{1}{2}\eta_{\vartheta,C_\vartheta D_{\vartheta}^m}\, |\, \vartheta>\tau_i\right\},
 $$
which is well-defined by induction hypothesis. 
Then \Cref{lem:xalpha} guarantees that
there exists $D=D_{\tau_1}=\ldots = D_{\tau_s}$ and $C=C_{\tau_1}=\ldots = C_{\tau_s}$, such that $Q_p(\eta_{\tau_i,C_{\tau_i}D_{\tau_i}^{m}})
>\exp(-D_{\tau_i}^2)$. Thus we get the case 
$h=\ell+1$, and the claim follows by induction.
\end{proof}

\begin{remark}
$D_{\tau}$ will grow fairly fast with respect to $h(\tau)$, i.e.\ we have $D_\theta\gg R_\theta\gg D_\tau \gg R_\tau$ for $\theta<\tau$. 
Therefore the lattice we find on $X$ will be much more refined on a higher stratum compared to the lattice on a lower stratum. 
However, we obtain a larger transversality region for lattice points from a lower stratum with much smaller amount of transversality. 
\end{remark}

\begin{remark}
Another natural order of induction is from the bottom stratum, as those minimal elements $\xi\in \CrS(X)$ are necessarily smooth manifolds after passing to the reduced version $\xi_{\mathrm{R}}$. Then we can get some transversality in a neighborhood of $X_{\xi}$. However, to make sure the achieved transversality is not destroyed when we work on a higher stratum, we are forced to use a larger $D$ on the higher stratum. Then \Cref{prop:lattice} only guarantees such lattice outside a neighborhood of $X_{\xi}$, which may exceed the region where we have transversality from the induction assumption. In other words, we have another numerical question on existence of $\eta,D,R$ for each stratum with different restrictions. However, in this case, the conditions are working against each other making the existence unclear.  
\end{remark}

We are now ready to give a proof of \Cref{thm:existence_intro} on the existence of asymptotically holomorphic sequences of sections quantitatively transverse to the zero section.

\begin{proof}[Proof of \Cref{thm:existence_intro}]
 By \Cref{prop:lattice} and \Cref{prop:size}, for each singular stratum $\tau$ and $k\gg 0$, we can find a lattice $\Lambda_{\tau}$ on the complement of a neighborhood of the singular locus in $\tau$, such that the balls of $g_k$-radius $R_{\tau}$ around $\Lambda_{\tau}$ cover $X$. We start with any asymptotically holomorphic section $s_k$. 
We use the sections $s_{k,x}$ of \Cref{lemma:localsection}. For the main stratum $\tau_0=\tau_{\max}$, we perturb 
 $$
  s_k':=s_k+\sum_{x\in \Lambda_{\tau}}w_{x}s_{k,x}\, ,
  $$
applying the same argument of \cite[Proposition 23]{Do96} for $k\gg 0$.
That is, there are $w_x$ for $x\in \Lambda_{\tau_0}$ such that 
$s'_k$ is 
$\eta_{\tau_{0},C_{\tau_0} D_{\tau_{0}}^{m}}$-transverse 
to $0$ on the balls of $g_k$-radius $R_{\tau_{0}}=1$ around $\Lambda_{\tau_0}$.

Next we proceed similarly by induction on the height, i.e.\ perturbing the coefficients for lower strata. 
The numbers from \Cref{prop:size} ensure that each perturbation does not destroy the transversality we obtained from the previous step on the higher strata, so that the argument from \cite[Proposition 23]{Do96} can be applied stratum by stratum. 
As a consequence, we get a universal constant $\eta$, such that $s_k$ is 
$\eta$-transverse for $k\gg 0$.
\end{proof}

\begin{remark}\label{rmk:higher_derivatives}
By \Cref{rmk:higher,rmk:even_distribution}, for $N>0$ there are $C_N>0$ so that $|\nabla^m s_k|<C_N$ for any $m\le N$ and all $k$.
\end{remark}

In fact, the proof above also shows the following:

\begin{corollary}
For $k\gg 0$, we can assume that the $\eta$-transverse asymptotically holomorphic sections $s_k$ pullback to $\eta$-transverse asymptotically holomorphic sections on $\tau$, 
for any $\tau\in \CrS(X)$. As a special case, if there is a point $x\in X$ such that the $x$ is the only fixed point of the isotropy group of $x$ in a local chart, then we can assure that $x\notin s_k^{-1}(0)$.
\end{corollary}

\subsection*{The quotient case}

Let $(M,\omega)$ be a smooth symplectic manifold and
$G$ a finite group of symplectomorphisms. Then the global quotient $X=M/G$ is a symplectic orbifold in a natural
way, with an induced orbifold symplectic form $\omega_X$
such that $\omega=\pi^*\omega_X$, where $\pi:M\to X$ is
the natural projection. If we take an orbifold  compatible almost complex $J_X$ on $(X,\omega_X)$,
then $J=\pi^*J_X$ is an almost complex structure on $M$ compatible with $\omega$.
In particular, $G$ acts by isometries for the metric associated to $\omega$ and $J$. 

Suppose now that $[\omega/2\pi]\in H^2(M,\ZZ)$. Then
as $H^2(X,\QQ)=H^2(M,\QQ)^G$, we have that
$[\omega_X/2\pi]\in H^2(X,\QQ)$. After taking a positive
integer multiple, we can assume that 
$[\omega_X/2\pi]\in H^2(X,\ZZ)$. We will assume this is the case, keeping the same notations. Then there is a 
complex line bundle $L_X \to X$ with $c_1(L_X)=[\omega_X/2\pi]$. The pull-back $L=\pi^*L_X \to M$ is a complex line bundle with $c_1(L)=[\omega/2\pi]$.

Take an orbifold connection $\nabla_X$ on $L_X$, with 
curvature $F_{\nabla_X}=-i\omega_X$. This produces operators $\bd_X$ and $\bbd_X$. The pull-back connection $\nabla$ on $L$ has $F_\nabla=-i\omega$, and pull-back operators
$\bd$ and $\bbd$ are $G$-invariant. Hence we have naturally a correspondence 
 $$
 C^\infty_{\orb}(X, L_X^{\ox k})
 \cong C^\infty(M,L^{\ox k})^G\, ,
 $$
and the orbifold asymptotically $J_X$-holomorphic sections on $X$
correspond to $G$-invariant asymptotically 
$J$-holomorphic sections on $M$.

Theorem \ref{thm:existence_intro} readily produces the following corollary:

\begin{corollary}
\label{cor:existence_G_invariant}
For $k\gg 0$, there exists an asymptotically holomorphic sequence of sections $s_k$ of $L^{\otimes k}$ on $M$ that is $\eta$-transverse to $0$, for some $\eta>0$ independent of $k$, and it is invariant by the action of $G$.
In particular, $s_k^{-1}(0)$ is a symplectic submanifold invariant by $G$.
\end{corollary}

\section{Lefschetz hyperplane theorem for symplectic orbifolds} \label{sec:lefschetz}

Consider a compact almost complex orbifold $(X,J,\omega)$ with integer symplectic form
and let $Z_k=Z(s_k)\subset X$ be asymptotically holomorphic suborbifolds, constructed as
zero sets of asymptotically holomorphic sections $s_k$ of $L^{\ox k}$, for $k\gg 0$, where the complex line bundle
$L\to M$ has $c_1(L)=[\omega/2\pi]$.
Then the topology of $X$ determines to large extent the topology of $Z_k$. This is given
by an extension of the Lefschetz hyperplane theorem for the situation at hand.

First, we say that a smooth orbifold function $f:X\to \RR$ is \emph{Morse} if the critical points
are isolated and non-degenerate \cite{Hep}. We recall more precisely the relevant notions. Let $X$
be a $m$-dimensional orbifold. If $x\in X$ and
$(U,\tilde U,H,\varphi)$ is an orbifold chart around $x$, and $\tilde f:\tilde U\to \RR$ is a representative
of $f$, which is $H$-equivariant. 
Then $x$ is a \emph{critical point} if $d\tilde f(x)=0$. Note that if
the action of $H <\OO(m)$ is irreducible, or generally if $(T_x\tilde U)^H =0$, this implies automatically
that $d\tilde f(x)=0$. At a critical point, there is a well-defined notion of Hessian, given as
 $$
  H_f(x)= \left( \frac{\bd \tilde f}{\bd x_i\bd x_j}(x)\right).
  $$
The critical point is \emph{non-degenerate} if $H_f(x)$ is a non-degenerate symmetric bilinear form. Note that always $(\Sym^2 T_x\tilde U)^H \neq 0$ (at least it contains the scalar product),
  therefore a critical point can always be non-degenerate.
 Moreover, if a critical point is non-degenerate, then it is isolated in the critical set. Therefore a 
 Morse function has finitely many critical points (for a compact orbifold).

Let now $s_k$ be an asymptotically holomorphic sequence of sections. Assume $s_k$ is $\eta$-transverse to zero. By \Cref{rmk:higher_derivatives}, the $C^2$-norm of $s_k$ is universally bounded. As a consequence, there exists universal positive constants $c,C_-,C_+$, such that  $C_-d_k(x,Z_k)\le |s_k(x)|\le C_+ d_k(x,Z_k)$ for $x$ in a $g_k$-neighbourhood $B_{g_k}(Z_k,c)$ of $g_k$-radius $c>0$. We may assume that $c$ is small enough, such that $c\, C_+<\eta$. 
As a consequence, we have $|s_k(x)|<\eta$ on $B_{g_k}(Z_k,c)$. Then $\eta$-transversality implies that $|\nabla s_k|\ge \eta$ on  $B_{g_k}(Z_k,c)$. 

We now consider the following function, defined on $X\setminus Z_k$,
  $$
  f_k=\log |s_k|^2 .
  $$
Note that 
 $$
 df_k = \frac{2}{|s_k|^2} \la s_k,\nabla s_k \ra.
 $$
 Hence the critical points of $f_k$ are the critical points of $s_k$,  and they are in $U_k:=X\backslash B_{g_k}(Z_k,c)$.

\begin{proposition} \label{prop:ddd}
 Let $(X,J,\omega)$ be as above. Take a sequence of asymptotically holomorphic
 sections $s_k$ of $L^{\ox k}$ which is $\eta$-transverse to zero. Then 
 there exists another sequence of asymptotically holomorphic sections $s_k'$, $c\,\eta$-transverse for some $0<c<1$, such that
  $Z(s_k')=Z(s_k)=Z_k$ and  $f_k'=\log |s_k'|^2$ is orbifold Morse on $X \setminus Z_k$.
 \end{proposition}
 
 \begin{proof}
It is proved in \cite[Theorem 6.10]{Hep} that Morse functions are dense in the space of smooth orbifold functions defined on $X$ with its $C^\infty$-topology. 
We can perturb $f_k$ to a nearby $f_k'$ with a 
perturbation compactly supported in $U_k$ so that $f_k'$ is Morse and $|\nabla^p (f'_k-f_k)|\leq \epsilon k^{-1/2}$ for $p\le 2$.
Now consider 
 $$
 s_k':=e^{(f_k'-f_k)/2}  s_k\, ,
 $$
extended as $s_k$ on $B_{g_k}(Z_k,c)$.
This gives a well-defined section of $L^{\ox k}$, still denoted by $s_k'$, satisfying
$f_k'=\log |s_k'|^2$.

The first claim is that $s_k'$ is asymptotically holomorphic. On $B_{g_k}(Z_k,c)$, this is obvious.
On $U_k$, we will check that $|\nabla^p (s_k'-s_k)| \leq C_0\epsilon k^{-1/2}$,
for some universal $C_0>0$ and $p\le 2$, from which asymptotically holomorphicity then follows.
By \Cref{rmk:higher_derivatives}, the former in turns follows from the following
bounds
 $$
  |\nabla^p (e^{(f_k'-f_k)/2}-1)|<C\epsilon k^{-1/2}.
 $$
For $p=0$, the facts that $e^x -1 \leq 2x$ for $x$ small and $|f_k'-f_k|\leq \epsilon k^{-1/2}$
give $|e^{(f_k'-f_k)/2}-1| \leq 2\epsilon k^{-1/2}$. For $1\le p \le 2$, we use
that $|\nabla^p (f'_k-f_k)|\leq \epsilon k^{-1/2}$ to get 
$ |\nabla^p (e^{(f_k'-f_k)/2}-1)|  \leq C\epsilon k^{-1/2}$, where $C>0$ is universal.

Lastly, $Z(s_k')=Z(s_k)$ and $s_k'$ is $c\,\eta$-transverse for some $0<c<1$, thus concluding.
\end{proof}

As explained before \Cref{prop:ddd}, we have $|\nabla s_k|\ge \eta $ on $B_{g_k}(Z_k,c)$. As a consequence, we have that $B_{g_k}(Z_k,c)$ is diffeomorphic to a tubular neighborhood of $Z_k$ in $L^{\otimes k}|_{Z_k}$. To see this, for $c$ sufficiently small and $x\in  B_{g_k}(Z_k,c)$, one can parallel transport $s_k(x)$ to $L^{\otimes k}|_{Z_k}$ over the unique length minimizing geodesic from $x$ to $Z_k$. Then the lower bounds $|\nabla s_k|\ge \eta $ on $B_{g_k}(Z,c)$ implies that such map is a diffeomorphism onto the image for some small enough $c$. In particular, $Z_k$ is a deformation retract of $B_{g_k}(Z,c)$.

We are now in place to prove \Cref{thm:homology_lefschetz_hyperplane_intro} on the relationship between the homology groups of the Donaldson submanifolds and those of the ambient manifold.
 
\begin{proof}[Proof of \Cref{thm:homology_lefschetz_hyperplane_intro}]
 We start by changing $s_k$ for a section such that $f_k=\log |s_k|^2$ is Morse with
 \Cref{prop:ddd}.
 All critical points are on $U_k=M\setminus B_{g_k}(Z_k,c)$. 
Denote  $S_t:=f_k^{-1}(-\infty,t]$. Take $m\ll 0$ so that $S_m \subset B_{g_k}(Z_k,c)$,
and moreover, $S_m$ is a tubular neighbourhood of
$Z_k$, diffeomorphic to $B_{g_k}(Z_k,c)$ (this can be arranged as previously discussed). This means in particular that $H_i(S_m,\RR)\cong H_i(Z_k,\RR)$.
 We can moreover arrange
 that each critical level $H_\lambda=f_k^{-1}(\lambda)$ contains only one
 critical point $p_0$. Take $a<\lambda<b$ so that $p_0$ is the only
 critical point in the set $f_k^{-1}([a,b])$. 
 We are going to see that $H_i(S_a,\RR)\cong H_i(S_b,\RR)$ for $i\leq n-2$, and that there
 is a surjection $H_{n-1}(S_b,\RR)\twoheadrightarrow H_{n-1}(S_a,\RR)$. 
 Once this is established, one can then reason inductively one critical point after the other (in the order given by their value via $f_k$), by starting at the sub-level $S_m$ and ending at the sub-level $X=S_{M}$, where 
 $M$ denotes the maximum of $f_k$, thus concluding the proof of the statement.
 Let us then now prove the isomorphism/surjection between the homologies.

Let $(U,\tilde U,H)$ be a chart centered at $p_0$, where we consider
$\tilde U=B_r(0)$, and $H<\OO(2n)$. The tangent space $T_{p_0}X$ identifies
naturally with $\RR^{2n}$, where $\tilde U\subset \RR^{2n}$, as an $H$-representation.
The Hessian $H_{f_k}(0)$ decomposes $T_{p_0}X = T^+_{p_0}X\oplus  T^-_{p_0}X$, where
both are $H$-representations, and the Hessian is positive/negative definite on $T_{p_0}^\pm X$.
 
Let us call, following \cite[Definition 4.5]{Hep},  \emph{index of $p_0$} the representation $\mathrm{ind}_{p_0}:=T^-_{p_0}X$. 
By \cite[Theorem 7.6]{Hep}, $S_b$ has the homotopy type of $S_a$ with a copy
of $\DD(\mathrm{ind}_p)/H$ glued to $H_{a}=\bd S_a$ along $\bd \DD(\mathrm{ind}_p)/G$.
Here $\DD(\mathrm{ind}_p)$ is the unit disc
in $\mathrm{inc}_p$.

We show shortly that $m:=\dim (\mathrm{ind}_{p_0})\geq n$. Given that, we have that $S_b$ has the
homotopy type of $S_a$ with a copy $\DD^m/H$ glued along $S^{m-1}/H$. There is then an exact sequence
 $$
  \ldots \to 
  H_{i+1}(\DD^m/H,S^{m-1}/H,\RR) 
  \to H_i(S_a,\RR) \to H_i(S_b,\RR) \to
  H_{i}(\DD^m/H,S^{m-1}/H,\RR)   \to \ldots
  $$
As $H_i(\DD^m/H,S^{m-1}/H,\RR)=0$ for $i\neq m$ and it is equal to $\RR$ for $i=m$, we get that
 $H_i(S_a,\RR) \cong H_i(S_b,\RR)$ for $i<m-1$, and that, for $i=m-1$, there is a surjection 
 $H_i(S_a,\RR) \twoheadrightarrow H_i(S_b,\RR)$. As $m\geq n$, we get that 
   $H_i(S_a,\RR) \cong H_i(S_b,\RR)$ for $i\leq n-2$ and a surjection 
 $H_i(S_a,\RR) \twoheadrightarrow H_i(S_b,\RR)$ for $i=n-1$.
 
 To check that $m\geq n$, suppose by contradiction that $H_{f_k}$ is positive definite on 
 a subspace $P\subset T_{p_0}X$ of dimension $>n$. Then $\Pi(x)=H_{f_k}(x)+H_{f_k}(Jx)$
 would be positive definite on a non-zero complex subspace $P\cap JP$. Let us check
 that $\Pi$ must be negative definite on $T_{p_0}X$, giving a contradiction.
 
 The Hessian is $H_{f_k}=\nabla df_k$. Therefore $\Pi$ kills the $(2,0)$ and $(0,2)$-parts,
 only retaining the $(1,1)$-part. Moreover, at a critical point, the Hessian coincides
 with the usual derivative. Therefore $\Pi(x)=\bar\bd \bd f_k(x)$ (see the proof of \cite[Proposition 39]{Do96}),
 by choosing coordinates with $N_J(x)=0$.  Now
  \begin{align*}
 \bd f_k =& \frac{1}{|s_k|^2} (\la \bd_{L^{\ox k}} s_k,s_k\ra + \la s_k,\bar \bd_{L^{\ox k}} s_k \ra), \\
  \bar\bd \bd f_k =&  \frac{1}{|s_k|^4}
   (\la \bd_{L^{\ox k}} s_k,s_k\ra + \la s_k,\bar \bd_{L^{\ox k}} s_k \ra)^2 + \\
    &+ 
  \frac{1}{|s_k|^2} (\la \bar\bd_{L^{\ox k}} \bd_{L^{\ox k}} s_k,s_k\ra + \la \bd_{L^{\ox k}} s_k, \bd_{L^{\ox k}} s_k \ra 
  +  \la \bar \bd_{L^{\ox k}} s_k,\bar \bd_{L^{\ox k}} s_k\ra + \la s_k,\bd_{L^{\ox k}}\bar \bd_{L^{\ox k}} s_k \ra).
  \end{align*}
At a critical point of $f_k$, we have that $\nabla s=0$, which implies that $\bd_{L^{\ox k}} s_k =\bar \bd_{L^{\ox k}} s_k =0$. Hence the first as well as the linear terms of the second line of the right hand side drop out. 
Lastly, $(\bar \bd_{L^{\ox k}} \bd_{L^{\ox k}} + \bd_{L^{\ox k}}\bar \bd_{L^{\ox k}}) s_k = ik\omega s_k $, hence, at a critical point $x$ of $f_k$,
$$ 
 \bar\bd \bd f_k (x)=ik\omega + 
  \frac{1}{|s_k|^2} ( - \la\bd_{L^{\ox k}} \bar\bd_{L^{\ox k}} s_k,s_k\ra  
+ \la s_k,\bd_{L^{\ox k}}\bar \bd_{L^{\ox k}} s_k \ra ).
$$

The asymptotically holomorphic conditions say that 
$|\nabla \bar\bd_{L^{\ox k}} s_k| \leq C_1 k^{-1/2}$. There is also a lower bound on $|s_k|$ 
with respect to the $g_k$-metric. 
So with respect to the $g_k$-metric
$$ 
 \bar\bd \bd f_k (x)=i\omega_k + O(k^{-1/2})
 $$
hence it is negative definite for $k$ large enough since $i\omega_k(u,Ju)=-||u||_{g_k} <0$.
\end{proof}

Now we move to the homotopy groups. First, we deal with the orbifold fundamental group (see \Cref{def:orb_fundamental}).

\begin{theorem}
\label{thm:lefschetz_hyperplane_fund_grp}
If $n\geq 3$ then there is an isomorphism $\pi_1^\orb(X)\cong \pi_1^\orb(Z_k)$  for $k\gg 0$. If $n=2$ then there is
a surjection $\pi_1^\orb(Z_k)\twoheadrightarrow\pi_1^\orb(X)$ for $k\gg 0$.
\end{theorem}

\begin{proof}
We consider the Morse function $f_k(x)$ and take 
$m\ll 0$ so that $S_m:=\{x \in X| \, f_k(x)\leq m\}$
is a tubular neighbourhood of $Z_k$, diffeomorphic to one of
the form $B_{g_k}(Z_k,c_0)$. This is an orbifold  bundle with fiber the disc $B(\nu_x,c_0)$, where $\nu \to Z_k$ is the orbifold normal bundle. 
Now, collapsing radially this orbifold bundle gives an orbifold deformation retract onto $Z_k$, i.e.\ the radial collapsing map $B(\nu_x,c_0)\to Z_k\subset B(\nu_x,c_0)$ is homotopic, as an orbifold map (i.e.\ in the sense of \cite{Chen06bis}), to the identity of $B(\nu_x,c_0)$.
Now, according to \cite[Theorem 1.2.(2)]{Chen06bis}, this means that $\pi_1^\orb(S_m) \simeq \pi_1^\orb(Z_k)$.
Now we move the value of $f_k$, and check
that the level sets $S_a$ have all the same
$\pi_1^\orb(S_a)$, for $n\geq 3$ (and there are epimorphisms for $n=2$). 
Eventually, for the maximum
$m'$ of $f_k(x)$, we have $\pi_1^\orb(X)=\pi_1^\orb(S_{m'})
\cong \pi_1^\orb(S_m)\cong \pi_1^\orb(Z_k)$.

We use the same notation as in the proof of \Cref{thm:homology_lefschetz_hyperplane_intro}. In particular, we analyze the case where $p_0$ is the unique critical point for $f_k$ having value in the interval $[a,b]$. We use the notation $P$ and $D$ of \Cref{def:orb_fundamental}.

Consider first the case where $p_0$ which is not singular, i.e.\ $H_{p_0}=\{1\}$ or equivalently $p_0\not \in P\cup D$. We have a Seifert-Van Kampen theorem with 
 $\pi_1(S_a\setminus(P\cup D))$, $\pi_1(\DD^m)=\{1\}$ and the intersection $\pi_1(S^{m-1})=\{1\}$ when $m\ge n \ge 3$, yielding
 $\pi_1(S_b\setminus(P\cup D))$. 
 So $\pi_1(S_b\setminus(P\cup D))\cong \pi_1(S_a\setminus(P\cup D))$, and there is no new $D_i$, so 
 $\pi_1^\orb(S_b)\cong \pi_1^\orb(S_a)$.  Similarly, when $n=m=2$, we have $\pi_1^\orb(S_a)\to \pi_1^\orb(S_b)$ is surjective.

Next, suppose that $p_0$ is singular with $H_{p_0}=H<\OO(2n)$, i.e.\ $p_0\in P\cup D$. We use $\widetilde{P},\widetilde{D}$ to denote the preimage of $P,D$ under the quotient map $\RR^{2n}\to \RR^{2n}/H$ in the local model $H\ltimes \RR^{2n}$ near $p_0$, which are unions of linear subspaces. Then there is an isomorphism $\pi_1((\DD^m\setminus (\widetilde{P}\cup \widetilde{D}))/H) \cong
\pi_1((S^{m-1}\setminus (\widetilde{P}\cup \widetilde{D}))/H)$, because 
when removing the central point, everything retracts to the boundary equivariantly. Therefore Seifert-Van Kampen theorem gives an isomorphism  $\pi_1(S_a\setminus(P\cup D)) \cong \pi_1(S_b\setminus(P\cup D))$.

Note that in $\DD^{m}$, $\widetilde{D}$ is a union of subspaces of codimension at most $2$. Now, if $p_0\notin D$, then $D$ must be disjoint from the stable manifold of $p_0$, hence the submanifolds $D_i\cap S_b$ retract onto $D_i\cap S_a$ under the Morse flow. That is every $D_i$ in $S_b$ has nontrivial intersection with $S_a$. Hence the quotients of $\pi_1(S_a\setminus(P\cup D)) \cong \pi_1(S_b\setminus(P\cup D))$ by the $\gamma_i$ are also the same, i.e.\ we get an isomorphism $\pi_1^\orb(S_a)\cong \pi_1^\orb(S_b)$.
Next we assume $p_0\in D$. Let $I$ be the index set such that $i\in I$ if and only if $p_0\in D_i$. Note that in a local model $H\ltimes \RR^{2n}$, 
$D_i=\mathrm{Fix}(H_i)/N_H(H_i)$ for some subgroup $H_i< H$ with $\mathrm{Fix}(H_i)\subset \RR^{2n}$ 
of codimension $2$. Note that the lifting $\widetilde{f}_k$ of $f_k$ to $\RR^{2n}$ has the property that $0$ is a critical point of $\widetilde{f}_k|_{\mathrm{Fix}(H_i)}$. Then  $D_i \cap S_a\neq \emptyset$ unless this is a minimum. In the former case, we can argue as before and we obtain again an isomorphism, as we quotient by 
the same relation on  $\pi_1(S_a\setminus(P\cup D)) \cong \pi_1(S_b\setminus(P\cup D))$. In the latter case, when it is a minimum,
it must be that $\mathrm{Fix}(H_i)\subset T^+_{p_0}X$. Hence the Morse index of $p_0$ is at most $2$. If $n\geq 3$ this is
impossible and we are done. If $n=2$, then at least we have an epimorphism $\pi_1^\orb(S_b)\twoheadrightarrow \pi_1^\orb(S_a)$, since in the second space we quotient by an extra $\gamma_i^{m_i}$. 
\end{proof}

For the homotopy groups, we assume that the 
ambient orbifold is simply-connected.

\begin{theorem}
\label{thm:homotopy_lefschetz_hyperplane}
Let $(X,\omega)$ be a symplectic orbifold of dimension $2n$ with $n\geq 3$, which 
has $\pi_1^{\orb}(X)=\{1\}$,
and let $Z_k$ be Donaldson suborbifolds for $k\gg 0$. Then 
 $$
   \begin{aligned}
   \pi_i(Z_k)\ox \RR \cong \pi_i(X)\ox \RR, \text{ for }i\leq n-2, \\
  \pi_i(Z_k)\ox \RR \surj \pi_i(X)\ox \RR, \text{ for }i= n-1.
  \end{aligned}
$$
\end{theorem}

\begin{proof}
From Theorem \ref{thm:lefschetz_hyperplane_fund_grp}, we have $\pi_1^{\orb}(Z_k)\cong \pi_1^{\orb}(X)=\{1\}$, for $n\geq 3$. As there is a surjection $\pi_1^{\orb}(X)\twoheadrightarrow \pi_1(X)$, we have that both $X$ and $Z_k$
are simply-connected. 
Now the result follows from \Cref{equivalence1}, by 
recalling that $V^i_X=(\pi_i(X)\ox \RR)^*$, and that
monomorphisms become epimorphisms after dualization.
\end{proof}

\begin{remark} 
The isomorphisms and monomorphism in \Cref{thm:homotopy_lefschetz_hyperplane} hold
over $\QQ$ as well.
\end{remark}

\section{Hard Lefschetz property} \label{sec:hard_lefschetz_property}

In this section we recall the $s$-Lefschetz property for any
compact symplectic manifold, generalizing the hard Lefschetz
property and study it for Donaldson suborbifolds
of  symplectic orbifolds.

\begin{definition} \label{s-Lefschetz}
  Let $(X,\omega)$ be a compact symplectic orbifold of dimension $2n$. We
  say that $X$ is $s$-Lefschetz with $s\leq n-1$ if
  $$
   [\omega]^{n-i}: H^i(X,\RR) \too H^{2n-i}(X,\RR)
  $$
  is an isomorphism for all $i\leq s$. 
\end{definition}

Note that $X$ is $(n-1)$-Lefschetz if it satisfies the hard
Lefschetz theorem. 

\begin{theorem} \label{Lefschetz-Donaldson}
Let $X$ be a compact symplectic orbifold of dimension $2n$, and
let $Z\hookrightarrow X$ be a Donaldson suborbifold. Then, for each $s\leq
n-2$, $X$ is $s$-Lefschetz if and only if $Z$ is $s$-Lefschetz.
\end{theorem}

\begin{proof}
For any orbifold differential form $x$ on $X$, we shall denote by $\hat{x}$
the differential form on $Z$ given by $\hat{x}=j^*(x)$, where $j^*$ is the restriction map induced by the inclusion
$j\colon Z\hookrightarrow M$.  Let now
$p=2(n-1)-i$, where $i \leq {n-2}$, and let 
us focus on
$j^*\colon H^p(X,\RR)\to H^p(Z,\RR)$. Then, for $[z]\in H^p(X,\RR)$, we claim that
 \begin{equation} \label{eqn:v4}
  j^*[z]=0  \iff  [z]\cup [\omega]=0.
 \end{equation}
 
This can be seen via Poincar\'{e} duality as follows. Clearly $j^*[z]=0$ if and
only if for any $a\in H^i(Z,\RR)$ we have $j^*[z]\cdot a=0$. We know that
there is an isomorphism $H^i(Z,\RR)\cong H^i(X,\RR)$ (as $i\leq n-2$), thus we can
assume that there is a closed $i$-form $x$ on $X$ with
$[x|_Z]=[\hat{x}]=a$. So
 $$
 j^*[z]\cdot [\hat{x}]=\int_Z \hat{z}\wedge \hat{x} =\int_X z\wedge  x \wedge
 k\,\omega,
 $$
 since $[Z]=k \PD [\omega]$. Hence $j^*[z]=0$ if and only if
 $[z\wedge \omega] \cdot [x]=0$ for all $[x]\in
 H^i(X,\RR)$, from where the claim follows.

Now suppose that $X$ is $s$-Lefschetz, so
$[\omega]^{n-i}:H^i(X,\RR)\to H^{2n-i}(X,\RR)$ is an isomorphism for
$i\leq s$. We want to check that the map
$[\omega_Z]^{n-1-i}:H^i(Z,\RR) \to
 H^{2n-2-i}(Z,\RR)$ is injective. Let $[\hat{z}]\in H^i(Z,\RR) \cong H^i(X,\RR)$ and
 extend it to $[z] \in H^i(X,\RR)$. Then, $[\omega_Z]^{n-1-i} [\hat{z}]=0$
 implies that $j^*[\omega^{n-1-i} \wedge z]=0$, which by~(\ref{eqn:v4})
 is equivalent to
 $[\omega^{n-1-i} \wedge z\wedge \omega]=0$. Using the
 $s$-Lefschetz property of $X$, we get $[z]=0$ and thus
 $[\hat{z}]=0$.

 The converse is easy. If $X$ is $s$-Lefschetz and we take $[z]\in
 H^i(X,\RR)$ such that $[\omega^{n-i}\wedge z]=0$, from~(\ref{eqn:v4})
 it follows that
 $j^*[\omega^{n-1-i} \wedge z]=0$, i.e., $[\omega_Z^{n-1-i} \wedge
 z|_Z]=0$.  Hence $[\hat{z}]=0$ in $H^i(Z,\RR)$ and so $[z]=0$ since
 $i\leq n-2$.
\end{proof}

The following more precise version of \Cref{hardLefschetzDonaldson_intro} is then a direct consequence of the above:

\begin{corollary} \label{hardLefschetzDonaldson}
Let $X$ be a compact symplectic orbifold of dimension $2n$, and
let $Z\subset X$ be a Donaldson suborbifold. If $X$ is
hard Lefschetz, then $Z$ is also hard Lefschetz. Moreover, $X$ is
$(n-2)$-Lefschetz (but not necessarily hard Lefschetz) if and only if $Z$ is hard Lefschetz.
\end{corollary}

\section{Formality of Donaldson suborbifolds}\label{sec:formality}

In this section we recall the concept of $s$-formality. First, we need
some definitions and results about minimal models.
Let $(A,d)$ be a {\it differential graded algebra} (in the sequel,
we shall say just a differential algebra), that is, $A$ is a
graded commutative algebra over $\RR$, with a differential $d$ which is a derivation, i.e.
$d(a\cdot b) = (da)\cdot b +(-1)^{\deg (a)} a\cdot (db)$, where
$\deg(a)$ is the degree of $a$.
A differential algebra $(A,d)$ is said to be {\it minimal\/} if it satisfied the following two properties.
\begin{enumerate}
 \item $A$ is free as an algebra, that is, $A$ is the free
 algebra $\bigwedge V$ over a graded vector space $V=\oplus V^i$.
 \item There exists a collection of generators $\{ a_\tau,
 \tau\in I\}$, for some well ordered index set $I$, such that, for any $\mu,\tau\in I$, 
 $\deg(a_\mu)\leq \deg(a_\tau)$ if $\mu < \tau$, and $d
 a_\tau$ is expressed in terms of the the preceding $a_\mu$ (i.e.\ of the $a_\mu$ with $\mu<\tau$).
 This implies that $da_\tau$ does not have a linear part, i.e., it
 lives in ${\bigwedge V}^{>0} \cdot {\bigwedge V}^{>0} \subset {\bigwedge V}$.
\end{enumerate}

Morphisms between differential algebras are required to be degree
preserving algebra maps which commute with the differentials.
Given a differential algebra $(A,d)$, we denote by $H^*(A)$ its
cohomology.
We say that $({\SM},d)$ is a {\it minimal model} of the
differential algebra $(A,d)$ if $({\SM},d)$ is minimal and there
exists a morphism of differential graded algebras $\rho\colon
{({\SM},d)}\longrightarrow {(A,d)}$ inducing an isomorphism
$\rho^*\colon H^*({\SM})\longrightarrow H^*(A)$ on cohomology.
In \cite{Hal83} Halperin proved that any differential algebra
$(A,d)$ with $H^0(A)=\RR$ has a minimal model unique up to isomorphism. 

A minimal model $({\SM},\rd)$ is said to be {\it formal} if there is a
morphism of differential algebras $\psi\colon {({\SM},d)}\longrightarrow
(H^*({\SM}),d=0)$ that induces the identity on cohomology. 

A {\it minimal model\/} of a connected orbifold $X$ is a
minimal model $(\bigwedge V,d)$ for the de Rham complex $(\Omega_{\orb}(X),d)$ of orbifold forms on $X$. If $X$ is simply connected then the dual of the real homotopy vector space
$\pi_i(X)\otimes \RR$ is isomorphic to $V^i$ for any $i$. This relation also happens when $i>1$ and $M$ is nilpotent, that is,
the fundamental group $\pi_1(X)$ is nilpotent and its action on
$\pi_j(X)$ is nilpotent for $j>1$ (see~\cite{DGMS,GMo}).

We say that $X$ is {\it formal\/} if its minimal model is
formal or, equivalently, the differential algebras $(\Omega_{\orb}(X),d)$
and $(H^*(X,\RR),d=0)$ have the same minimal model. (For details
see~\cite{DGMS,GMo,Lup90} for example.) Therefore, if $X$ is formal and
simply connected, then the real homotopy groups $\pi_i(X)\otimes
\RR$ are obtained from the minimal model of $(H^*(X,\RR),d=0)$.

The following notion has been introduced in \cite{FM}:
\begin{definition}\label{primera}
Let $({\SM},d)$ be a minimal model. We say that
$({\SM},d)$ is \emph{$s$-formal}
($s\geq 0$) if we can write ${\SM}={\bigwedge V}$ such that for each $i\leq s$
the space $V^i$ of generators of degree $i$ decomposes as a direct
sum $V^i=C^i\oplus N^i$, where the spaces $C^i$ and $N^i$ satisfy
the three following conditions:
\begin{enumerate}
\item $d(C^i) = 0$,
\item the differential map $d\colon N^i\longrightarrow \bigwedge V$ is
injective,
\item any closed element in the ideal
$I_s=I(\bigoplus\limits_{i\leq s} N^i)$, generated by the space
$\bigoplus\limits_{i\leq s} N^i$ in the free algebra $\bigwedge
(\bigoplus\limits_{i\leq s} V^i)$, is exact in $\bigwedge V$.
\end{enumerate}
\end{definition}

In what follows, we shall write $N^{\leq s}$ and $\bigwedge
V^{\leq s}$ instead of $\bigoplus\limits_{i\leq s}N^i$ and
$\bigwedge(\bigoplus\limits_{i\leq s}V^i)$, respectively. In
particular, $I_s=N^{\leq s}\cdot (\bigwedge V^{\leq s})$.

A connected orbifold $M$ is $s$-formal if its minimal model is $s$-formal.

The following result is proved in \cite{FM} for compact differentiable manifolds, but the proof follows verbatim for orbifolds.

\begin{theorem}[{\cite[Theorem 3.1]{FM}}]
\label{criterio2}
Let $X$ be a connected and orientable compact orbifold of dimension $2n$, or $(2n-1)$. Then $X$ is formal if and
only if it is $(n-1)$-formal.
\end{theorem}

Let $X$ and $Y$ be compact orbifolds. We say that
an orbifold map $f\colon X \to Y$ is a cohomology
$s$-equivalence ($s\geq 0$) if it induces isomorphisms
$f^*\colon H^i(Y,\RR) \stackrel{\cong}{\longrightarrow} H^i(X,\RR)$ on cohomology for $i<s$, and a
monomorphism $f^*\colon H^s(Y,\RR)\hookrightarrow H^s(X,\RR)$ for $i=s$. Therefore the inclusion
$Z\subset X$ of a Donaldson suborbifold is a cohomology $(n-1)$-equivalence by \Cref{thm:homology_lefschetz_hyperplane_intro}.

For a cohomology $s$-equivalence we have the following:

\begin{proposition}[{\cite[Proposition 5.1]{FM}}]
\label{equivalence1}
Let $X$ and $Y$ be compact orbifolds and let $f\colon X \to Y$ be
a cohomology  $s$-equivalence. Then there exist minimal models
$({\bigwedge V}_X,d)$ and $({\bigwedge V}_Y,d)$ of $X$ and $Y$, respectively, such
that $f$ induces a morphism of differential algebras $F\colon
(\bigwedge V_Y^{\leq s},d) \to (\bigwedge V_X^{\leq s},d)$ where
$F: V_{Y}^{<s} \congm V_{X}^{<s}$ is an isomorphism and $F: V_{Y}^s
\subset V_{X}^s$ is a monomorphism.
\end{proposition}

This result  is stated in \cite{FM} for manifolds but works equally for orbifolds. 

The following is a more precise version of \Cref{thm:formality}. 

\begin{theorem}\label{equivalence2}
\begin{enumerate}
\item Let $X$ and $Y$ be compact orbifolds, and let $f \colon X
\to Y$ be a cohomology $s$-equivalence. If $Y$ is $(s-1)$-formal
then $X$ is $(s-1)$-formal. 
\item Let $X$ be a compact symplectic
orbifold of dimension $2n$ and let $Z\subset X$ be a Donaldson
suborbifold. For each $s\leq n-2$, if $M$ is $s$-formal then $Z$
is $s$-formal. In particular, $Z$ is formal if $X$ is
$(n-2)$-formal.
\end{enumerate}
\end{theorem}

\begin{proof}
Let $(\bigwedge V_X,d)$ and $(\bigwedge V_Y,d)$ be the minimal
models of $X$ and $Y$, respectively, constructed in
Proposition~\ref{equivalence1}. For $i<s$, decompose $V_Y^i=C_Y^i
\oplus N_Y^i$ satisfying the conditions of Definition
\ref{primera}. Then, taking into account
Proposition~\ref{equivalence1}, we set $V_X^i=C_X^i \oplus N_X^i$
under the natural isomorphism  $F\colon V_Y^i\cong V_X^i$, $i<s$.
Consider a closed element $F(\eta)=\hat{\eta} \in N_X^{<s}\cdot
(\bigwedge V_X^{<s})$. Hence $\eta$ is a closed element in
$N_Y^{<s}\cdot (\bigwedge V_Y^{<s})$ and, by the
$(s-1)$-formality of $Y$, it is exact, i.e., $\eta=d\xi$, for
$\xi\in \bigwedge V_Y$. Take the image $\hat{\eta}=d(F(\xi))$ in
$\bigwedge V_X$. This proves (1). Now (2) follows from (1) and
using that the inclusion $j\colon Z\hookrightarrow X$ is a cohomology
$(n-1)$-equivalence.
\end{proof}


\bibliographystyle{alpha} 
\bibliography{bib-orbifold-ah}

\end{document}